\documentclass{amsart}
\usepackage{a4}
\usepackage{graphics}
\usepackage{color}
\usepackage{amssymb}
\usepackage[all]{xy}
\usepackage{amsmath}
\usepackage{stmaryrd}
\usepackage{subfigure}
\usepackage{longtable}
\usepackage{setspace}
\usepackage{booktabs}

\usepackage{color}
\CompileMatrices

\newtheorem{theorem}{Theorem}[section]
\newtheorem{lemma}[theorem]{Lemma}
\newtheorem{proposition}[theorem]{Proposition}
\newtheorem{corollary}[theorem]{Corollary}

\theoremstyle{definition}

\newtheorem{definition}[theorem]{Definition}
\newtheorem{example}[theorem]{Example}

\newtheorem{setting}[theorem]{Setting}
\newtheorem{construction}[theorem]{Construction}

\newtheorem{remark}[theorem]{Remark}
\theoremstyle{remark}

\numberwithin{equation}{section}
\def\Rho{{\rm R}}

\def\div{{\rm div}}

\def\quot{/\!\!/}

\def\rq#1{\widehat{#1}}

\def\b#1{\overline{#1}}
\def\bangle#1{\langle #1 \rangle}

\def\KK{{\mathbb C}}
\def\TT{{\mathbb T}}
\def\ZZ{{\mathbb Z}}

\def\QQ{{\mathbb Q}}
\def\PP{{\mathbb P}}

\def\Mov{{\rm Mov}}

\def\trop{{\rm trop}}
\def\Aut{\operatorname{Aut}}

\def\Cl{\operatorname{Cl}}

\def\Pic{\operatorname{Pic}}

\def\Spec{{\rm Spec}}

\def\conv{{\rm conv}}
\def\cone{{\rm cone}}

\def\Aut{{\rm Aut}}
\def\rk{{\rm rk}\,}
\def\relint{{\rm relint}}

\begin{document}
\title[On terminal Fano $3$-folds with $2$-torus action]%
{On terminal Fano $3$-folds with $2$-torus action}
\author[B.~Bechtold]{Benjamin Bechtold} 
\address{Mathematisches Institut, Universit\"at T\"ubingen,
Auf der Morgenstelle 10, 72076 T\"ubingen, Germany}
\email{benjamin.bechtold@googlemail.com}
\author[J.~Hausen]{J\"urgen Hausen} 
\address{Mathematisches Institut, Universit\"at T\"ubingen,
Auf der Morgenstelle 10, 72076 T\"ubingen, Germany}
\email{juergen.hausen@uni-tuebingen.de}
\author[E.~Huggenberger]{\\Elaine Huggenberger} 
\address{Mathematisches Institut, Universit\"at T\"ubingen,
Auf der Morgenstelle 10, 72076 T\"ubingen, Germany}
\email{elaine.huggenberger@uni-tuebingen.de}
\author[M.~Nicolussi]{Michele Nicolussi} 
\address{Mathematisches Institut, Universit\"at T\"ubingen,
Auf der Morgenstelle 10, 72076 T\"ubingen, Germany}
\email{michele.nicolussi@uni-tuebingen.de}

\begin{abstract}
We classify the terminal 
$\QQ$-factorial Fano threefolds
of Picard number one that come with an 
effective action of a two-dimensional torus.
Our approach applies also to higher
dimensions and generalizes the correspondence 
between toric Fano varieties and lattice 
polytopes:
to any Fano variety with a 
complete intersection Cox ring we associate 
its ``anticanonical complex'', which is a certain
polyhedral complex living in the lattice 
of one parameter groups of an ambient toric 
variety.
For resolutions constructed via the tropical 
variety, the lattice points inside the anticanonical 
complex control the discrepancies.
This leads, for example, to simple 
characterizations of terminality and canonicity.
\end{abstract}


\maketitle


\section{The main results}
\label{sec:intro}

This article contributes to the classification 
of Fano threefolds, that means normal projective 
algebraic varieties~$X$ of dimension three with 
an ample anticanonical divisor;
we work over the field $\KK$ of complex numbers.
Whereas the smooth Fano threefolds are well known
due to Iskovskikh~\cite{Is1,Is2} and 
Mori/Mukai~\cite{MoMu},
the singular case is still widely open.
We restrict to terminal singularities, i.e.,~the 
mildest class in the context of the minimal model 
program.
Let $\TT \subseteq \Aut(X)$ be a maximal torus.
If $\dim(\TT) = \dim(X)$ holds, then $X$ is 
a toric Fano variety and the classification
can be performed in the setting of lattice 
polytopes, see~\cite{BoBo,Ka}.
We go one step further and consider torus actions 
of complexity one, meaning that we have 
$\dim(\TT) = \dim(X)-1$.
Our approach is via the Cox ring 
$$ 
\mathcal{R}(X) \ := \ \bigoplus_{\Cl(X)} \Gamma(X,\mathcal{O}(D)),
$$
which can be associated to any normal complete 
variety $X$ with finitely generated divisor class 
group $\Cl(X)$; see~\cite[Sec.~1.4]{ArDeHaLa} for 
the details of this definition.
For a Fano variety $X$ with at most terminal 
singularities, $\Cl(X)$ is finitely 
generated~\cite[Sec.~2.1]{IsPr}.
If, in addition, $X$ comes with a torus action of 
complexity one, then $X$ is rational, 
the Cox ring $\mathcal{R}(X)$ is finitely generated,
uniquely determines $X$, and admits an explicit 
description as a complete intersection~\cite{HaHe,HaSu}.
Our main result gives a classification of the 
terminal $\QQ$-factorial threefolds of Picard 
number one by listing their Cox rings.

\begin{theorem}
\label{thm:classif}
The following table lists the Cox rings 
$\mathcal{R}(X)$ of the non-toric terminal 
$\QQ$-factorial Fano threefolds~$X$ of 
Picard number one with an effective two-torus 
action;
the $\Cl(X)$-degrees of the 
generators $T_1,\ldots,T_r$ are denoted as columns 
$w_i \in \Cl(X)$ of a matrix $[w_1,\ldots,w_r]$.
Additionally we give the selfintersection number $(-\mathcal{K}_X)^3$
for the anticanonical class $-\mathcal{K}_X\in\Cl(X)$
and the Gorenstein index $\iota(X)$, i.e.,
the smallest positive integer such that $\iota(X)\cdot\mathcal{K}_X$ 
is Cartier.

{\small
\begin{longtable}[htbp]{cccccc}
\toprule
No.
&
$\mathcal{R}(X)$
&
$\Cl(X)$
&
$[w_1,\ldots, w_r]$
&
$(-\mathcal{K}_X)^3$
&
$\iota(X)$
\\
\midrule
1
&
$\frac{\KK[T_1,\ldots,T_5]}{\bangle{T_1T_2+T_3T_4+T_5^2}}$
&
$\ZZ$
&
$
\left[
\begin{smallmatrix}
1 & 1 & 1 & 1 & 1
\end{smallmatrix}
\right]
$
&
$54$
&
$1$
\\
\midrule
2
&
$\frac{\KK[T_1,\ldots,T_5]}{\bangle{T_1T_2+T_3T_4+T_5^2}}$
&
$\ZZ$
&
$
\left[
\begin{smallmatrix}
1 & 5 & 2 & 4 & 3
\end{smallmatrix}
\right]
$
&
$729/20$
&
$20$
\\
\midrule
3
&
$\frac{\KK[T_1,\ldots,T_5]}{\bangle{T_1T_2+T_3T_4+T_5^2}}$
&
$\ZZ\oplus\ZZ/5\ZZ$
&
$
\left[
\begin{smallmatrix}
1 & 1 & 1 & 1 & 1\\
\b{2} & \b{3} & \b{1} & \b{4} & \b{0}
\end{smallmatrix}
\right]
$
&
$54/5$
&
$5$
\\
\midrule
4
&
$\frac{\KK[T_1,\ldots,T_5]}{\bangle{T_1T_2+T_3T_4+T_5^3}}$
&
$\ZZ$
&
$
\left[
\begin{smallmatrix}
1 & 5 & 3 & 3 & 2
\end{smallmatrix}
\right]
$
&
$512/15$
&
$15$
\\
\midrule
5
&
$\frac{\KK[T_1,\ldots,T_5]}{\bangle{T_1T_2+T_3T_4+T_5^4}}$
&
$\ZZ$
&
$
\left[
\begin{smallmatrix}
1 & 3 & 2 & 2 & 1
\end{smallmatrix}
\right]
$
&
$125/3$
&
$6$
\\
\midrule
6
&
$\frac{\KK[T_1,\ldots,T_5]}{\bangle{T_1T_2+T_3T_4+T_5^4}}$
&
$\ZZ\oplus\ZZ/2\ZZ$
&
$
\left[
\begin{smallmatrix}
1 & 3 & 2 & 2 & 1\\
\b{1} & \b{1} & \b{1} & \b{1} & \b{0}
\end{smallmatrix}
\right]
$
&
$125/6$
&
$12$
\\
\midrule
7
&
$\frac{\KK[T_1,\ldots,T_5]}{\bangle{T_1T_2+T_3T_4+T_5^6}}$
&
$\ZZ$
&
$
\left[
\begin{smallmatrix}
2 & 4 & 3 & 3 & 1
\end{smallmatrix}
\right]
$
&
$343/12$
&
$12$
\\
\midrule
8
&
$\frac{\KK[T_1,\ldots,T_5]}{\bangle{T_1T_2+T_3^2T_4+T_5^2}}$
&
$\ZZ$
&
$
\left[
\begin{smallmatrix}
1 & 3 & 1 & 2 & 2
\end{smallmatrix}
\right]
$
&
$125/3$
&
$6$
\\
\midrule
9
&
$\frac{\KK[T_1,\ldots,T_5]}{\bangle{T_1T_2+T_3^2T_4+T_5^2}}$
&
$\ZZ$
&
$
\left[
\begin{smallmatrix}
1 & 5 & 2 & 2 & 3
\end{smallmatrix}
\right]
$
&
$343/10$
&
$10$
\\
\midrule
10
&
$\frac{\KK[T_1,\ldots,T_5]}{\bangle{T_1T_2+T_3^2T_4+T_5^2}}$
&
$\ZZ$
&
$
\left[
\begin{smallmatrix}
3 & 7 & 4 & 2 & 5
\end{smallmatrix}
\right]
$
&
$1331/84$
&
$84$
\\
\midrule
11
&
$\frac{\KK[T_1,\ldots,T_5]}{\bangle{T_1T_2+T_3^2T_4+T_5^3}}$
&
$\ZZ$
&
$
\left[
\begin{smallmatrix}
2 & 1 & 1 & 1 & 1
\end{smallmatrix}
\right]
$
&
$81/2$
&
$2$
\\
\midrule
12
&
$\frac{\KK[T_1,\ldots,T_5]}{\bangle{T_1T_2+T_3^2T_4+T_5^3}}$
&
$\ZZ$
&
$
\left[
\begin{smallmatrix}
3 & 3 & 1 & 4 & 2
\end{smallmatrix}
\right]
$
&
$343/12$
&
$12$
\\
\midrule
13
&
$\frac{\KK[T_1,\ldots,T_5]}{\bangle{T_1T_2+T_3^2T_4+T_5^3}}$
&
$\ZZ\oplus\ZZ/3\ZZ$
&
$
\left[
\begin{smallmatrix}
2 & 1 & 1 & 1 & 1\\
\b{1} & \b{2} & \b{1} & \b{1} & \b{0}
\end{smallmatrix}
\right]
$
&
$27/2$
&
$6$
\\
\midrule
14
&
$\frac{\KK[T_1,\ldots,T_5]}{\bangle{T_1T_2+T_3^2T_4+T_5^6}}$
&
$\ZZ$
&
$
\left[
\begin{smallmatrix}
3 & 3 & 2 & 2 & 1
\end{smallmatrix}
\right]
$
&
$125/6$
&
$6$
\\
\midrule
15
&
$\frac{\KK[T_1,\ldots,T_5]}{\bangle{T_1T_2+T_3^2T_4^2+T_5^2}}$
&
$\ZZ\oplus\ZZ/2\ZZ$
&
$
\left[
\begin{smallmatrix}
1 & 3 & 1 & 1 & 2\\
\b{1} & \b{1} & \b{0} & \b{0} & \b{1}
\end{smallmatrix}
\right]
$
&
$64/3$
&
$6$
\\
\midrule
16
&
$\frac{\KK[T_1,\ldots,T_5]}{\bangle{T_1T_2+T_3^2T_4^2+T_5^3}}$
&
$\ZZ$
&
$
\left[
\begin{smallmatrix}
3 & 3 & 2 & 1 & 2
\end{smallmatrix}
\right]
$
&
$125/6$
&
$6$
\\
\midrule
17
&
$\frac{\KK[T_1,\ldots,T_5]}{\bangle{T_1T_2+T_3^3T_4+T_5^2}}$
&
$\ZZ$
&
$
\left[
\begin{smallmatrix}
1 & 3 & 1 & 1 & 2
\end{smallmatrix}
\right]
$
&
$128/3$
&
$3$
\\
\midrule
18
&
$\frac{\KK[T_1,\ldots,T_5]}{\bangle{T_1T_2+T_3^3T_4+T_5^2}}$
&
$\ZZ$
&
$
\left[
\begin{smallmatrix}
2 & 4 & 1 & 3 & 3
\end{smallmatrix}
\right]
$
&
$343/12$
&
$12$
\\
\midrule
19
&
$\frac{\KK[T_1,\ldots,T_5]}{\bangle{T_1T_2+T_3^3T_4+T_5^2}}$
&
$\ZZ$
&
$
\left[
\begin{smallmatrix}
3 & 7 & 2 & 4 & 5
\end{smallmatrix}
\right]
$
&
$1331/84$
&
$84$
\\
\midrule
20
&
$\frac{\KK[T_1,\ldots,T_5]}{\bangle{T_1T_2+T_3^3T_4+T_5^2}}$
&
$\ZZ\oplus\ZZ/2\ZZ$
&
$
\left[
\begin{smallmatrix}
1 & 3 & 1 & 1 & 2\\
\b{1} & \b{1} & \b{0} & \b{0} & \b{1}
\end{smallmatrix}
\right]
$
&
$64/3$
&
$6$
\\
\midrule
21
&
$\frac{\KK[T_1,\ldots,T_5]}{\bangle{T_1T_2+T_3^3T_4+T_5^4}}$
&
$\ZZ\oplus\ZZ/2\ZZ$
&
$
\left[
\begin{smallmatrix}
2 & 2 & 1 & 1 & 1\\
\b{1} & \b{1} & \b{1} & \b{1} & \b{0}
\end{smallmatrix}
\right]
$
&
$27/2$
&
$4$
\\
\midrule
22
&
$\frac{\KK[T_1,\ldots,T_5]}{\bangle{T_1T_2+T_3^3T_4^2+T_5^2}}$
&
$\ZZ$
&
$
\left[
\begin{smallmatrix}
3 & 5 & 2 & 1 & 4
\end{smallmatrix}
\right]
$
&
$343/15$
&
$30$
\\
\midrule
23
&
$\frac{\KK[T_1,\ldots,T_5]}{\bangle{T_1T_2+T_3^3T_4^3+T_5^2}}$
&
$\ZZ$
&
$
\left[
\begin{smallmatrix}
2 & 4 & 1 & 1 & 3
\end{smallmatrix}
\right]
$
&
$125/4$
&
$4$
\\
\midrule
24
&
$\frac{\KK[T_1,\ldots,T_5]}{\bangle{T_1T_2+T_3^5T_4+T_5^2}}$
&
$\ZZ$
&
$
\left[
\begin{smallmatrix}
2 & 4 & 1 & 1 & 3
\end{smallmatrix}
\right]
$
&
$125/4$
&
$4$
\\
\midrule
25
&
$\frac{\KK[T_1,\ldots,T_5]}{\bangle{T_1T_2+T_3^6T_4+T_5^2}}$
&
$\ZZ$
&
$
\left[
\begin{smallmatrix}
3 & 5 & 1 & 2 & 4
\end{smallmatrix}
\right]
$
&
$343/15$
&
$30$
\\
\midrule
26
&
$\frac{\KK[T_1,\ldots,T_6]}{\bangle{T_1T_2+T_3T_4+T_5^2,\lambda T_3T_4+T_5^2+T_6^2}}$
&
$\ZZ\oplus\ZZ/2\ZZ$
&
$
\left[
\begin{smallmatrix}
1 & 1 & 1 & 1 & 1 & 1\\
\b{1} & \b{1} & \b{0} & \b{0} & \b{1} & \b{0}
\end{smallmatrix}
\right]
$
&
$16$
&
$2$
\\
\midrule
27
&
$\frac{\KK[T_1,\ldots,T_5]}{\bangle{T_1T_2T_3+T_4^3+T_5^2}}$
&
$\ZZ$
&
$
\left[
\begin{smallmatrix}
1 & 1 & 4 & 2 & 3
\end{smallmatrix}
\right]
$
&
$125/4$
&
$4$
\\
\midrule
28
&
$\frac{\KK[T_1,\ldots,T_5]}{\bangle{T_1T_2T_3+T_4^3+T_5^2}}$
&
$\ZZ$
&
$
\left[
\begin{smallmatrix}
2 & 3 & 1 & 2 & 3
\end{smallmatrix}
\right]
$
&
$125/6$
&
$6$
\\
\midrule
29
&
$\frac{\KK[T_1,\ldots,T_5]}{\bangle{T_1T_2+T_3^3+T_4^2}}$
&
$\ZZ$
&
$
\left[
\begin{smallmatrix}
1 & 5 & 2 & 3 & 1
\end{smallmatrix}
\right]
$
&
$216/5$
&
$5$
\\
\midrule
30
&
$\frac{\KK[T_1,\ldots,T_5]}{\bangle{T_1T_2+T_3^3+T_4^2}}$
&
$\ZZ$
&
$
\left[
\begin{smallmatrix}
1 & 5 & 2 & 3 & 2
\end{smallmatrix}
\right]
$
&
$343/10$
&
$10$
\\
\midrule
31
&
$\frac{\KK[T_1,\ldots,T_5]}{\bangle{T_1T_2+T_3^3+T_4^2}}$
&
$\ZZ$
&
$
\left[
\begin{smallmatrix}
1 & 5 & 2 & 3 & 3
\end{smallmatrix}
\right]
$
&
$512/15$
&
$15$
\\
\midrule
32
&
$\frac{\KK[T_1,\ldots,T_5]}{\bangle{T_1T_2+T_3^3+T_4^2}}$
&
$\ZZ$
&
$
\left[
\begin{smallmatrix}
1 & 5 & 2 & 3 & 4
\end{smallmatrix}
\right]
$
&
$729/20$
&
$20$
\\
\midrule
33
&
$\frac{\KK[T_1,\ldots,T_5]}{\bangle{T_1T_2+T_3^4+T_4^2}}$
&
$\ZZ\oplus\ZZ/2\ZZ$
&
$
\left[
\begin{smallmatrix}
1 & 3 & 1 & 2 & 1\\
\b{1} & \b{1} & \b{0} & \b{1} & \b{0}
\end{smallmatrix}
\right]
$
&
$64/3$
&
$6$
\\
\midrule
34
&
$\frac{\KK[T_1,\ldots,T_5]}{\bangle{T_1T_2+T_3^4+T_4^2}}$
&
$\ZZ\oplus\ZZ/2\ZZ$
&
$
\left[
\begin{smallmatrix}
1 & 3 & 1 & 2 & 2\\
\b{1} & \b{1} & \b{0} & \b{1} & \b{1}
\end{smallmatrix}
\right]
$
&
$125/6$
&
$12$
\\
\midrule
35
&
$\frac{\KK[T_1,\ldots,T_5]}{\bangle{T_1T_2+T_3^5+T_4^2}}$
&
$\ZZ$
&
$
\left[
\begin{smallmatrix}
3 & 7 & 2 & 5 & 1
\end{smallmatrix}
\right]
$
&
$512/21$
&
$21$
\\
\midrule
36
&
$\frac{\KK[T_1,\ldots,T_5]}{\bangle{T_1T_2+T_3^5+T_4^2}}$
&
$\ZZ$
&
$
\left[
\begin{smallmatrix}
3 & 7 & 2 & 5 & 4
\end{smallmatrix}
\right]
$
&
$1331/84$
&
$84$
\\
\midrule
37
&
$\frac{\KK[T_1,\ldots,T_5]}{\bangle{T_1T_2+T_3^6+T_4^2}}$
&
$\ZZ\oplus\ZZ/2\ZZ$
&
$
\left[
\begin{smallmatrix}
2 & 4 & 1 & 3 & 1\\
\b{1} & \b{1} & \b{1} & \b{0} & \b{0}
\end{smallmatrix}
\right]
$
&
$125/8$
&
$8$
\\
\midrule
38
&
$\frac{\KK[T_1,\ldots,T_5]}{\bangle{T_1T_2+T_3^3+T_4^3}}$
&
$\ZZ\oplus\ZZ/3\ZZ$
&
$
\left[
\begin{smallmatrix}
1 & 2 & 1 & 1 & 1\\
\b{1} & \b{2} & \b{2} & \b{0} & \b{0}
\end{smallmatrix}
\right]
$
&
$27/2$
&
$6$
\\
\midrule
39
&
$\frac{\KK[T_1,\ldots,T_5]}{\bangle{T_1T_2+T_3^4+T_4^3}}$
&
$\ZZ$
&
$
\left[
\begin{smallmatrix}
5 & 7 & 3 & 4 & 1
\end{smallmatrix}
\right]
$
&
$512/35$
&
$35$
\\
\midrule
40
&
$\frac{\KK[T_1,\ldots,T_5]}{\bangle{T_1T_2+T_3^4+T_4^3}}$
&
$\ZZ$
&
$
\left[
\begin{smallmatrix}
5 & 7 & 3 & 4 & 2
\end{smallmatrix}
\right]
$
&
$729/70$
&
$70$
\\
\midrule
41
&
$\frac{\KK[T_1,\ldots,T_5]}{\bangle{T_1T_2+T_3^3T_4+T_5^4}}$
&
$\ZZ$
&
$
\left[
\begin{smallmatrix}
2 & 2 & 1 & 1 & 1
\end{smallmatrix}
\right]
$
&
$27$
&
$2$
\\
\midrule
42
&
$\frac{\KK[T_1,\ldots,T_5]}{\bangle{T_1T_2+T_3^4T_4+T_5^3}}$
&
$\ZZ$
&
$
\left[
\begin{smallmatrix}
3 & 3 & 1 & 2 & 2
\end{smallmatrix}
\right]
$
&
$125/6$
&
$6$
\\
\midrule
43
&
$\frac{\KK[T_1,\ldots,T_5]}{\bangle{T_1T_2+T_3^4T_4^2+T_5^3}}$
&
$\ZZ$
&
$
\left[
\begin{smallmatrix}
3 & 3 & 1 & 1 & 2
\end{smallmatrix}
\right]
$
&
$64/3$
&
$3$
\\
\midrule
44
&
$\frac{\KK[T_1,\ldots,T_5]}{\bangle{T_1T_2+T_3^5T_4+T_5^3}}$
&
$\ZZ$
&
$
\left[
\begin{smallmatrix}
3 & 3 & 1 & 1 & 2
\end{smallmatrix}
\right]
$
&
$64/3$
&
$3$
\\
\midrule
45
&
$\frac{\KK[T_1,\ldots,T_5]}{\bangle{T_1T_2+T_3^2+T_4^2}}$
&
$\ZZ\oplus\ZZ/2\ZZ$
&
$
\left[
\begin{smallmatrix}
2 & 4 & 3 & 3 & 1\\
\overline{1} & \overline{1} & \overline{1} & \overline{0} & \overline{0}
\end{smallmatrix}
\right]
$
&
$343/24$
&
$24$
\\
\midrule
46
&
$\frac{\KK[T_1,\ldots,T_5]}{\bangle{T_1T_2+T_3^3+T_4^2}}$
&
$\ZZ$
&
$
\left[
\begin{smallmatrix}
5 & 7 & 4 & 6 & 1
\end{smallmatrix}
\right]
$
&
$1331/70$
&
$70$
\\
\midrule
47
&
$\frac{\KK[T_1,\ldots,T_5]}{\bangle{T_1T_2+T_3^3+T_4^2}}$
&
$\ZZ$
&
$
\left[
\begin{smallmatrix}
5 & 7 & 4 & 6 & 3
\end{smallmatrix}
\right]
$
&
$2197/210$
&
$210$
\\
\bottomrule
\end{longtable}
}
\noindent
where $\lambda\in\KK^*\setminus\{1\}$ in No.~26.
Any two of the Cox rings $\mathcal{R}(X)$ 
listed in the table correspond to non-isomorphic 
varieties.
All corresponding varieties $X$ are rational
and No.~1  is the only smooth one.
\end{theorem}

Our approach works also in higher dimensions
and applies more generally to Fano varieties $X$ 
with a complete intersection Cox ring $\mathcal{R}(X)$.
For such varieties we introduce the anticanonical 
complex as a combinatorial data in the spirit of the 
Fano polytopes from toric geometry.
Theorem~\ref{thm:main} characterizes in particular 
canonical and terminal singularities in terms of 
lattice points inside the anticanonical complex.
Using the knowledge on the Cox ring of varieties
with a torus action of complexity one provided 
by~\cite{HaSu,HaHe}, we obtain an explicit
description of the anticanonical complex in 
that case, see Section~\ref{sec:cpl1-antican}.
This enables us to derive  in Section~\ref{sec:terminal3folds}
effective bounds for the defining data of the 
terminal Fano threefolds of Picard number one 
that come with an action 
of a two-dimensional torus.
One of the basic principles is to construct 
suitable lattice simplices via the anticanonical complex
and to apply the volume bounds obtained 
in~\cite{AvKrNi,Ka,Ka2}.
Having found reasonable bounds, the remaining 
step is to figure out the terminal ones from 
the list of possible candidates by means of
Theorem~\ref{thm:main}. 
This is done using the software package~\cite{MDS}, 
where among other things our methods are implemented, 
and finally leads to the list given in 
Theorem~\ref{thm:classif}.

We now present the construction of the  anticanonical 
complex and the characterization of singularities.
Consider a normal Fano variety~$X$ 
with divisor class group $\Cl(X)$ and Cox ring 
$\mathcal{R}(X)$.
Recall that $\mathcal{R}(X)$ is factorially 
$\Cl(X)$-graded, i.e.,~every homogeneous 
nonzero nonunit is a product of $\Cl(X)$-primes,
see~\cite[Sec.~3]{Ha2}.
We assume that  $\mathcal{R}(X)$ is a complete 
intersection in the sense that it comes with a 
presentation by $\Cl(X)$-homogeneous generators 
$T_\varrho$ and relations $g_i$:
$$
\mathcal{R}(X) 
\ = \ 
\KK[T_\varrho; \; \varrho \in \Rho] / \bangle{g_1,\ldots,g_s},
$$
where the meaning of the index set $\Rho$ becomes clear 
soon, the $T_\varrho$ define pairwise 
nonassociated $\Cl(X)$-primes in $\mathcal{R}(X)$ 
and the dimension of $\mathcal{R}(X)$ equals 
$\vert \Rho \vert - s$.
This setting leads to a closed embedding
$X \subseteq Z_\Sigma$ into a toric variety
$Z_\Sigma$ arising from a fan $\Sigma$,
where the divisor class group and
Cox ring of $Z_\Sigma$ are given~by
$$
\Cl(Z_\Sigma) \ \cong \ \Cl(X),
\qquad\qquad
\mathcal{R}(Z_\Sigma) 
\ = \ \KK[T_\varrho; \; \varrho \in \Rho];
$$
see~\cite[Constr.~3.13 and Prop.~3.14]{Ha2}. 
Removing successively closed torus orbits 
from $Z_\Sigma$, we can achieve that $X$ 
intersects every closed torus orbit of 
$Z_\Sigma$. We speak then of $X \subseteq Z_\Sigma$
as the minimal toric embedding.

We provide the necessary details for 
defining the anticanonical complex.
Consider the degree homomorphism
$Q \colon \ZZ^\Rho \to \Cl(X)$ sending the 
$\varrho$-th canonical basis vector 
$e_\varrho \in \ZZ^\Rho$ to $\deg(T_\varrho) \in \Cl(X)$
and let $P^* \colon \ZZ^n \to \ZZ^{\Rho}$ 
be a linear embedding with image $\ker(Q)$.
Then we have
$$
\Cl(Z_\Sigma) 
\ \cong \
\ZZ^\Rho/P^*(\ZZ^n) 
\ \cong  \
\Cl(X).
$$
Denote by $P \colon \ZZ^{\Rho} \to \ZZ^n$ 
the dual map of $P^*$.
Set $e_{\Sigma} := \sum e_\varrho$.
Then the canonical classes of $Z_\Sigma$ and $X$ 
are given as
$$ 
\mathcal{K}_{\Sigma}
\ = \
-Q(e_{\Sigma}),
\qquad\qquad
\mathcal{K}_X 
\ = \
\sum \deg(g_i)  + \mathcal{K}_{\Sigma}.
$$
The defining fan $\Sigma$ of $Z_\Sigma$ lives in the 
lattice $\ZZ^n$ and is obtained as follows.
Let $\gamma_{\Rho} \subseteq \QQ^{\Rho}$ be the positive
orthant, spanned by the $e_\varrho$,
and $e_X \in \ZZ^{\Rho}$ any representative of $\mathcal{K}_X$.
Then we have polytopes
$$
B(-\mathcal{K}_X) 
\ := \ 
Q^{-1}(-\mathcal{K}_X) \cap \gamma_{\Rho}
\ \subseteq \ 
\QQ^{\Rho},
\qquad\qquad
(P^*)^{-1}(B(-\mathcal{K}_X) + e_X)
\ \subseteq \
\QQ^n.
$$
The normal fan $\Sigma_c$ of the second polytope defines 
a toric variety $Z_c$ containing $X$ as a subvariety
and $\Sigma$ is the subfan of $\Sigma_c$ generated
by the cones that correspond to a torus orbit of $Z_c$
intersecting $X$.
In particular, the rays of $\Sigma$ have exactly 
the vectors $v_\varrho := P(e_\varrho) \in \ZZ^n$ as their 
primitive generators; we identify $\varrho \in \Rho$ 
with the ray through $v_\varrho$.

Let $\trop(X) \subseteq \QQ^n$ denote the 
tropical variety of $X \cap \TT$, endowed 
with a fan structure that refines 
the projected normal fan $P(\mathcal{N}(B))$ 
in $\QQ^n$ of the Minkowski sum 
$B := B(g_1) + \ldots + B(g_s)$
of the Newton polytopes $B(g_i)$ 
of the relations $g_i$, 
i.e.,~$B(g_i) \subseteq \QQ^\Rho$ is the 
convex hull over the exponent vectors 
of $g_i$.

\begin{definition}
The \emph{anticanonical polyhedron} of 
$X$ is the dual polyhedron $A_X \subseteq \QQ^n$ 
of the polytope
$$
B_X
\ := \ 
(P^*)^{-1}(B(-\mathcal{K}_X) + B - e_{\Sigma}) 
\ \subseteq \ 
\QQ^n.
$$
The \emph{anticanonical complex} of $X$ 
is the coarsest common refinement of polyhedral 
complexes
$$ 
A^c_X
\ := \ 
{\rm faces}(A_X) \sqcap \Sigma \sqcap \trop(X).
$$
The \emph{relative interior} of $A_X^c$ 
is the interior of its support 
with respect to the tropical variety $\trop(X)$.
\end{definition}

\begin{example}
The $E_6$-singular cubic surface 
$X = V(z_1z_2^2 + z_2z_0^2 + z_3^3) \subseteq \PP_3$ 
is invariant under the $\KK^*$-action
$$
t \cdot [z_0,\ldots,z_3] 
\ = \
[z_0,t^{-3}z_1,t^{3}z_2,tz_3]
$$
on $\PP_3$. The divisor class group and the Cox ring 
of the surface $X$ are explicitly given by 
$$ 
\Cl(X) \ = \ \ZZ,
\qquad\qquad
\mathcal{R}(X) 
\ = \ 
\KK[T_{1},T_{2},T_{3},T_{4}]/ \bangle{T_{1}T_{2}^{3}+T_{3}^{3}+T_{4}^{2}},
$$
where the $\Cl(X)$-degrees of $T_1$, $T_2$, $T_3$, $T_4$ 
are $3$, $1$, $2$, $3$.
The minimal ambient toric variety $Z_\Sigma$
is an open subset of $Z_c = \PP_{3,1,2,3}$
and the tropical variety in $\QQ^{3}$ is
$$
\trop(X)
\ = \ 
\cone(e_{1},\pm e_{3})
\cup 
\cone(e_{2},\pm e_{3})
\cup 
\cone(-e_{1}-e_{2}, \pm e_{3}),
$$
where $e_{i} \in \QQ^{3}$ is the $i$-th canonical basis vector.
The anticanonical polyhedron $A_X \subseteq \QQ^3$ has the vertices 
$$
(-3,-3,-2),\
(-1,-1,-1),\
(3,0,1),\
(0,2,1),\
(0,0,1),\
(0,0,-1/5).
$$
The anticanonical complex $A_{X}^{c} = A_{X} \sqcap \trop(X)$ lives
on the three cones of $\trop(X)$ and thus is of dimension two.
\begin{center}
\includegraphics{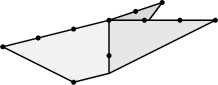}
\end{center}
\end{example}

Our aim is to characterize the behaviour 
of singularities of $X$ in terms of lattice 
points of the anticanonical complex $A_X^c$. 
Recall that for a $\QQ$-Gorenstein variety~$X$, 
that means that some non-zero multiple of the 
canonical divisor $\mathcal{K}_{X}$ is Cartier,
various types of singularities
are defined via the ramification formula
\begin{eqnarray*}
\mathcal{K}_{X'}
\ - \ 
\varphi^{*}(\mathcal{K}_{X})
& = &
\sum a_{i}E_{i}, 
\end{eqnarray*}
where $\varphi \colon X' \to X$ is a resolution,
the $E_{i}$ are the prime components of the 
exceptional divisor and the $a_{i}$ are called
the \emph{discrepancies} of the resolution.
One says that $X$ has at most 
\emph{log terminal ($\varepsilon$-log terminal for $0 < \varepsilon < 1$,
canonical, terminal)}
singularities, if for every resolution the
discrepancies $a_{i}$ satisfy $a_{i} > -1$
($a_{i} > -1+\varepsilon$, $a_{i} \ge 0$, $a_{i} > 0$).

We concern ourselves with Fano varieties $X$ that are
\emph{(strongly) tropically resolvable} in the sense that
some (every) subdivision of $\Sigma \sqcap \trop(X)$ admits a
regular refinement that induces a resolution of 
singularities $X' \to X$ with a suitable Mori dream space $X'$.
As we will see in Proposition~\ref{prop:cpl1isstr},
all normal rational varieties with a torus action 
of complexity one are strongly tropically resolvable.

\begin{theorem}
\label{thm:main}
Let $X$ be a (strongly) tropically resolvable
normal Fano variety with a complete intersection 
Cox ring.
\begin{enumerate}
\item
$A_X^c$ contains the origin in its relative
interior and all primitive generators of 
the fan $\Sigma$ are vertices of $A_X^c$.
\item 
$X$ has at most log terminal singularities 
if (and only if) the anticanonical complex 
$A_X^c$ is bounded.
\item
$X$ has at most $\varepsilon$-log terminal singularities
if (and only if) $0$ is the only lattice point in 
$\varepsilon A_{X}^{c}$.
\item
$X$ has at most canonical singularities
if (and only if) $0$ is the only 
lattice point in the relative interior of 
$A_X^c$.
\item
$X$ has at most terminal singularities
if (and only if) $0$ and the primitive generators 
$v_\varrho$ for $\varrho \in \Sigma^{(1)}$
are the only lattice points of $A_X^c$.
\end{enumerate}
\end{theorem}

Note that these statements generalize the corresponding 
characterizations of toric singularities in terms of 
lattice polytopes given for example in~\cite{BoBo}.  
In the toric case, i.e.,~in the absence of relations $g_i$, 
our anticanonical polytope $A_X$ is just the Fano polytope 
and the anticanonical complex is the subdivison of $A_X$ 
by the fan $\Sigma$.


\tableofcontents


We are grateful to Simon Keicher for supporting our 
work with his implementation of the anticanonical 
complex in~\cite{MDS} and for his very helpful advice 
concerning the computations.
Moreover, we would like to thank the 
referees for several very valuable suggestions.

\section{Discrepancies}
\label{sec:discr}

Here we prove Theorem~\ref{thm:main}.
The setting is the one introduced in 
Section~\ref{sec:intro}.
In particular, $X$ is a normal Fano 
variety with a complete intersection 
Cox ring $\mathcal{R}(X)$ 
given by $\Cl(X)$-homogeneous generators and 
relations and we have the associated 
minimal toric embedding:
$$
\mathcal{R}(X) 
\ = \ 
\KK[T_\varrho; \; \varrho \in \Rho] / \bangle{g_1,\ldots,g_s},
\qquad\qquad
X \subseteq Z_\Sigma.
$$

Assertion~(i) of Theorem~\ref{thm:main} 
holds under more general assumptions.
Therefore, we state and prove it separately. 
Note that we always have $0 \in \relint(A_X)$. 
For any ray $\varrho \in \Rho$ 
with $\varrho \nsubseteq A_X$ 
we denote by $v_{\varrho}'$
the intersection point  
of $\varrho$ and the boundary $\partial A_X$.

\begin{proposition}
Assume that the anticanonical class $-\mathcal{K}_X$ 
lies in the relative interior of the movable cone
of $X$.
Then, for every ray $\varrho \in \Sigma$, the 
primitive generator $v_{\varrho} \in \varrho$ 
is a vertex of $A_X$. 
In particular, $\varrho \nsubseteq A_{X}$ and we 
have $v'_\varrho = v_\varrho$.
\end{proposition}

\begin{proof}
By construction, the anticanonical polyhedron $A_X$ 
is the intersection of the half spaces 
$$
H_u \ := \ \{v \in \QQ^n; \; \bangle{u,v} \ge -1\},
\qquad
u \ \in \ 
B_{X}.
$$
Thus, our task is to show that for every 
ray $\varrho \in \Rho$ there is a facet
$B_{\varrho} \subseteq B_{X}$ with 
$\bangle{b,v_{\varrho}} = -1$ for all 
$b \in B_{\varrho}$.

Fix $\varrho \in \Rho$.
Then $v_\varrho = P(e_\varrho)$ holds with a unique
canonical basis vector $e_\varrho \in \ZZ^\Rho$.
Thus, for any $u \in \QQ^n$, we have 
$$ 
\bangle{u,v_\varrho} 
\ = \ 
\bangle{u,P(e_\varrho)}
\ = \ 
\bangle{P^*(u),e_\varrho}
$$
and $P^*(B_{X})$ equals $B(-\mathcal{K}_X)+B-e_{\Sigma}$. 
Since $B(-\mathcal{K}_X)$ and $B$ both lie in the
positive orthant $\gamma_{\Rho}$, 
we conclude $\bangle{u,v_\varrho} \ge -1$
for all $u \in B_X$.

Let $\gamma_{\varrho} \preceq \gamma_{\Rho}$ be the facet 
consisting of points with $\varrho$-th coordinate zero. 
The description of the movable cone given in~\cite[Prop.~4.1]{Ha2}
shows that $-\mathcal{K}_X$ lies in the relative interior 
of $Q(\gamma_\varrho)$. It follows that 
$$
B^{\varrho}(-\mathcal{K}_X) 
\ := \ 
B(-\mathcal{K}_X) \cap \gamma_{\varrho}
$$
is a facet of $B(-\mathcal{K}_X)$. 
Note that $\bangle{e,e_{\varrho}}=0$ holds for all
$e \in B^{\varrho}(-\mathcal{K}_X)$.
We claim that $B^{\varrho} := B \cap \gamma_{\varrho}$ is nonempty.
Indeed, since every $g_i$ is irreducible, it has 
an exponent $b_i \in B(g_i)$ with $\varrho$-th coordinate 
zero.
Thus $b_1+\ldots+b_s \in B^{\varrho}$ holds.
Note that we have $\bangle{e,e_{\varrho}}=0$ for all
$e \in B^{\varrho}$.
Since zero is the minimal possible value for 
linear forms from $B(-\mathcal{K}_X) + B$ on $e_\varrho$,
we see that $B^{\varrho}(-\mathcal{K}_X) + B^{\varrho}-e_{\Sigma}$
is a face of $B(-\mathcal{K}_X) + B -e_{\Sigma}$. 
By dimension reasons, it is a facet.
\end{proof}

Consider a toric modification
$Z_{\Sigma'} \to Z_\Sigma$ given by 
a subdivision $\Sigma' \to \Sigma$ 
of fans.
We introduce a shift of polynomials 
from
$\KK[T_\varrho; \; \varrho \in \Rho]$
to
$\KK[T_{\varrho'}; \; \varrho' \in \Rho']$,
where $\Rho \subseteq \Sigma$ and 
$\Rho' \subseteq \Sigma'$ are the sets of rays.
The toric Cox constructions $P \colon \ZZ^\Rho \to \ZZ^n$ 
and $P' \colon \ZZ^{\Rho'} \to \ZZ^n$ 
define homomorphisms of tori
$$ 
\xymatrix{
{\TT^{\Rho'}}
\ar[r]^{p'}
&
{\TT^{n}}
&
{\TT^{\Rho}}
\ar[l]_{p}
}.
$$
Let $g \in \KK[T_\varrho; \; \varrho \in \Rho]$
be without monomial factors.
The push-down of $g$ is the unique
$p_*(g) \in \KK[T_1,\ldots,T_n]$ without monomial factors 
such that $T^{\mu}p^*(p_*(g))= g$ holds for some 
Laurent monomial $T^\mu \in \KK[T_\varrho^{\pm 1}; \; \varrho \in \Rho]$.
The \emph{shift} of $g$ is the unique
$g' \in \KK[T_{\varrho'};\;\varrho' \in  \Rho']$ 
without monomial factors satisfying $p'_*(g') = p_*(g)$.

\begin{definition}
\label{def:tropres}
Let $X \subseteq Z_{\Sigma}$ be the minimal toric 
embedding of a complete variety given by a 
complete intersection Cox ring 
$$
\mathcal{R}(X) 
\ = \
\KK[T_{\varrho}; \; \varrho \in \Rho] / \bangle{g_1,\ldots,g_s}.
$$
\begin{enumerate}
\item
We call the modification $X' \to X$ arising from a subdivision
$\Sigma' \to \Sigma$ of fans a 
\emph{tropical resolution of singularities}
if $\Sigma'$ subdivides $\Sigma \sqcap \trop(X)$
and~$X'$ is smooth with complete intersection Cox ring
defined by the shifts $g_i'$ of $g_i$:
$$
\mathcal{R}(X') 
\ = \ 
\KK[T_{\varrho'}; \; \varrho' \in \Rho'] /\bangle{g'_1,\ldots,g'_s}.
$$ 
\item
We say that $X$ is \emph{strongly tropically resolvable} 
if every subdivision of $\Sigma \sqcap \trop(X)$
admits a regular refinement providing a tropical resolution 
of singularities. 
\end{enumerate}
\end{definition}

Assertions~(ii) to~(v) of Theorem~\ref{thm:main} 
will be obtained as a consequence of the 
following description of discrepancies 
of a tropical resolution. 

\begin{proposition}
\label{prop:discrep}
Let $\varphi \colon X' \to X$ be 
a tropical resolution of singularities
given by subdivision $\Sigma' \to \Sigma$
of fans.  
Then the discrepancy $\alpha_{\varrho}$ along
a divisor $D_{\varrho}$ corresponding to 
a ray $\varrho \in \Sigma'$ satisfies
$$
\alpha_{\varrho}
 =  
\frac{\Vert v_{\varrho} \Vert}{\Vert v'_{\varrho} \Vert} -1
\text{ if } \varrho \nsubseteq A_{X}^c,
\qquad
\alpha_{\varrho}
 \le  
-1 
\text{ if } \varrho \subseteq A_{X}^c.
$$
\end{proposition}

We provide two Lemmas used  in the proof 
of Proposition~\ref{prop:discrep} and 
also later.
The first one describes the exponents 
of a shifted polynomial.

\begin{lemma}
\label{lem:shift2newton}
Consider a subdivision of fans $\Sigma' \to \Sigma$
with Cox constructions given by
$P' \colon \ZZ^{\Rho'} \to \ZZ^n$ and $P \colon \ZZ^{\Rho} \to \ZZ^n$,
a polynomial   
$g = \sum a_\nu T^\nu \in \KK[T_\varrho; \; \varrho \in \Rho]$
without monomial factors 
and a linear surjection
$F \colon \QQ^{\Rho'} \to \QQ^{\Rho}$ 
with $P' = P \circ F$.
Then there is a unique $e_F \in \ZZ^{\Rho'}$
such that the shift $g'$ is given as  
$$
g' 
\ =  \ 
T^{e_F}\sum a_\nu T^{F^{*}(\nu)}
\ \in \ 
\KK[T_{\varrho'}; \; \varrho' \in \Rho'].
$$ 
In particular the exponents of $g$ (the vertices of $B(g)$)
correspond to the exponents of $g'$ (the vertices of $B(g')$).
Moreover, for any exponent $\nu$ of $g$,
the corresponding exponent $\nu'$ of $g'$ 
satisfies $\nu'_\varrho = \nu_\varrho$ for all 
$\varrho \in \Rho \subseteq \Rho'$.
\end{lemma}

\begin{proof}
Choose linear maps 
$\mu \colon \ZZ^{\Rho'} \to \ZZ^{\Rho'}$ and 
$\alpha \colon \ZZ^{\Rho'} \to \ZZ^\Rho$,
both of full rank, such that 
$F \circ \mu = \alpha$ holds.
Then $(p \circ \alpha)^*(p_*(g))$ 
equals $(p' \circ \mu)^*(p'_*(g'))$
which proves the displayed equality.
Choosing an $F$ given by a matrix $[E_r,F']$ 
with the $r \times r$ unit matrix $E_r$ 
gives the last statement.
\end{proof}

For a polynomial $g$, we denote by $\exp(g)$ the
set of its exponent vectors and, as earlier,
by $\mathcal{N}(B(g))$ the normal fan of its
Newton polytope.
Moreover, for a cone $\sigma \in \Sigma$ we denote 
by $\rq{\sigma} \in \rq{\Sigma}$ the unique cone 
with $P(\rq{\sigma}) = \sigma$.

\begin{lemma}
\label{lem:varinvert}
Let $h \in \KK[T_1,\ldots,T_n]$ and $e \in \ZZ^{\Rho}$ 
such that $g := T^e p^*(h)$ is a polynomial in 
$\KK[T_\varrho; \; \varrho \in \Rho]$ 
having no monomial factors. 
Consider a face $C \subseteq B(h)$,
the corresponding cone 
$\tau \in \mathcal{N}(B(h))$
and suppose that $\sigma \in \Sigma$ 
satisfies $\relint(\sigma) \subseteq \relint(\tau)$. 
Then we have
$$
\exp(g) \cap \rq{\sigma}^{\perp} 
\ = \ 
P^*(C \cap \exp(h)) + e . 
$$
\end{lemma}

\begin{proof}
To verify ``$\subseteq$'',
let $e_{g} \in \exp(g) \cap \rq{\sigma}^{\perp}$.
Then $e_{g} = P^{*}(u_{h})+e$ holds with some $u_{h} \in \exp(h)$.
Choose $\rq{v} \in \relint(\rq{\sigma})$.
Then, for any $u \in \exp(h)$, we have
$$
\bangle{u,P(\rq{v})}
\ = \ 
\bangle{P^{*}(u)+e,\rq{v}}
- 
\bangle{e,\rq{v}},
\qquad\qquad
\bangle{P^{*}(u)+e,\rq{v}} \ge 0.
$$
Moreover, $e_{g} \in \rq{\sigma}^{\perp}$
implies $\bangle{P^{*}(u_{h})+e,\rq{v}} = 0$.
Thus $u_{h} \in \exp(h)$ 
minimizes $P(\rq{v})$.
Since $P(\rq{v}) \in \relint(\tau)$ holds, 
we obtain $u_{h} \in C$.

For ``$\supseteq$'', let $u_{h} \in C \cap \exp(h)$.
Then $e_{g} := P^{*}(u_{h})+e$ lies in $\exp(g)$.
By monomial freeness of~$g$, we find that every 
ray $\varrho$ of $\sigma$ admits an
$\nu_\varrho \in \exp(g)$ with  
$\bangle{\nu_\varrho,e_\varrho} = 0$.
Write $\nu_\varrho = P^{*}(u_\varrho)+e$
with $u_\varrho \in \exp(h)$.
Then
$$
0 
= 
\bangle{\nu_\varrho, e_\varrho}
=
\bangle{u_\varrho, P(e_\varrho)}
+
\bangle{e,P(e_\varrho)}
\ge 
\bangle{u_{h},P(e_\varrho)}
+
\bangle{e,P(e_\varrho)}
=
\bangle{e_{g},e_\varrho}
\ge 
0
$$
holds, where the estimate in the middle is 
due to the fact that  $u_h$ minimizes 
$P(e_\varrho) \in \tau$.
In particular $e_{g}$ annihilates $\rq{\sigma}$.
\end{proof}

\begin{proof}[Proof of Proposition~\ref{prop:discrep}]
We write $\Rho \subseteq \Sigma$ and $\Rho' \subseteq \Sigma'$ 
for the respective sets of rays. 
The exceptional divisors of $\varphi \colon X' \to X$
are precisely the divisors $D_{X'}^{\varrho'}$ obtained 
as pullbacks of the toric divisors in $Z_\Sigma$ given 
by the rays $\varrho' \in \Rho' \setminus \Rho$; 
see~\cite[Prop.~3.14]{Ha2}.
We fix such $\varrho'$ 
and compute the discrepancy of $\varphi \colon X' \to X$ 
along $D_{X'}^{\varrho'}$.

Let $B := B(g_1) + \ldots + B(g_s)$ and
$B' := B(g_1') + \ldots + B(g_s')$ be the Minkowski
sums of the Newton polytopes $B(g_i)$ and $B(g_i')$.
The inverse image $P^{-1}(\varrho')$ is contained in a maximal 
cone $\tau \in \mathcal{N}(B(-\mathcal{K}_X) + B)$.
Let $\eta \in B(-\mathcal{K}_X) + B$ be the vertex 
corresponding to $\tau$.
Then $\eta = \nu_{-\mathcal{K}_X} + \nu$ with vertices 
$\nu_{-\mathcal{K}_X} \in B(-\mathcal{K}_X)$ and $\nu \in B$.
Moreover, we write $\nu' \in B'$ for the vertex corresponding 
to $\nu \in B$ in the sense of Lemma~\ref{lem:shift2newton}.

We compute the discrepancy of $\varphi \colon X' \to X$ 
along the divisor $D_{X'}^{\varrho'}$
using the following representatives of the 
canonical classes of $X$ and $X'$:
$$
D_{X}^{c}
\ := \ 
\sum_{\varrho \in \Rho} (-1 + \nu_{\varrho})D_X^{\varrho},
\qquad\qquad
D_{X'}^{c}
\ := \ 
\sum_{\varrho \in \Rho'} (-1 + \nu'_{\varrho})D_{X'}^{\varrho}.
$$
Note that by the definition of a tropical resolution of 
singularities, $D_{X'}^{c}$ is indeed a canonical divisor. 
Moreover, $D_{X'}^c - \varphi^*D_X^c$ is supported on 
the exceptional locus by Lemma~\ref{lem:shift2newton}.

Let $\sigma \in \Sigma$ be the cone with $\relint(\varrho')
\subseteq \relint(\sigma)$.
Then, on the corresponding chart $X_\sigma = X \cap Z_\sigma$,
the divisor $D_X^c$ is (rationally) principal.
More precisely, we claim that on  $X_\sigma$
this divisor has a presentation
$$
D_X^c 
\ = \ 
\frac{1}{m}\div({\chi^{mu}})
\qquad 
\text{with}
\quad 
u \ := \ (P^*)^{-1}(\nu_{-\mathcal{K}_X} +  \nu - e_\Sigma),
$$
where $m \in \ZZ_{>0}$ is such that $m u$ is integral
and $\chi^{mu}$ denotes the pullback of the toric character
function on $Z_\Sigma$ associated to $mu$.

To verify the claim, we first show 
$\bangle{\nu_{-\mathcal{K}_X}, e_{\varrho}} =0$ 
for all rays $\varrho$ of $\sigma$. 
Indeed, due to ampleness of the anticanonical class,
$B(-\mathcal{K}_X) \cap \relint(\rq{\sigma}^\perp \cap \gamma_\Rho)$ 
is non-empty, see~\cite[Prop.~4.1]{Ha2}, 
and thus contains some element $e^*$.
Because of $\relint(\varrho') \subseteq \relint(\sigma)$,
the preimage $P^{-1}(\varrho')$ contains a vector 
$\mu = \sum_{\varrho \in \sigma^{(1)}}{b_{\varrho} e_{\varrho}}$ 
with positive $b_{\varrho}$. We have $\bangle{e^*, \mu} =0$. 
Since $\nu_{-\mathcal{K}_X} \in B(-\mathcal{K}_X)$ 
is a minimizing vertex for $\mu$, we conclude 
$\bangle{\nu_{-\mathcal{K}_X}, \mu} = 0$
and hence $\bangle{\nu_{-\mathcal{K}_X}, e_{\varrho}} = 0$ 
for all rays $\varrho$ of $\sigma$. 
Consequently, on $X_\sigma$, we have 
$$ 
\frac{1}{m}\div(\chi^{mu})
\ = \
\sum_{\varrho \in \sigma^{(1)}}{\bangle{u, v_{\varrho}}D_X^{\varrho}} 
\ =  ß
\sum_{\varrho \in \sigma^{(1)}}{\bangle{P^*u, e_{\varrho}}D_X^{\varrho}} 
\ = \
\sum_{\varrho \in \sigma^{(1)}}{\bangle{\nu-e_\Sigma, e_{\varrho}}D_X^{\varrho}}.
$$

Using the presentation of $D_X^c$ on $X_\sigma$ 
just obtained, we see that 
the discrepancy $a_{\varrho'}$ of $\varphi \colon X' \to X$ 
along $D_{X'}^{\varrho'}$ is the multiplicity of 
$D_{X'}^{c} - \div(\chi^u)$ along $D_{X'}^{\varrho'}$ 
and thus is concretely given by
$$
a_{\varrho'} 
\ = \ 
-1 + \nu'_{\varrho'} - \bangle{u, v_{\varrho'}}.
$$

We show that $\nu'_{\varrho'} = 0$ holds. 
First note that $\nu = \nu_1+\ldots+\nu_s$,
where $\nu_i$ is an exponent vector of $g_i$. 
Let $\nu'_i$ be the corresponding 
exponent vector of $g'_i$ in the sense of 
Lemma~\ref{lem:shift2newton}.
Then $\nu'= \nu'_1+\ldots+ \nu'_s$. 
We claim that $\nu'_{i,\varrho'}=0$ for all $i=1,\ldots,s$. 
By definition, $\nu'_i$ lies in the face of $B(g'_i)$ 
which is cut out by $P'^{-1}(\varrho')$. 
Consequently, the corresponding exponent vector of
the pushed down equation $p_*(g_i)$ lies in the face of 
$B(p_*(g_i))$ that is cut out by $\varrho'$.
Lemma~\ref{lem:varinvert} applied to $\varrho'$ and $g'_i$ 
yields that $\nu'_i$ is orthogonal to ${\rq{\varrho}}'$, i.e.,~we
have $\nu_{i,\varrho'}=0$. 

To conclude the proof we have to evaluate $\bangle{u,v_{\varrho'}}$.
For this, consider the maximal cone 
$\sigma^\sharp \in \mathcal{N}(B_X)$ 
corresponding to the vertex $u \in B_X$. 
Then we have $\varrho' \subseteq \sigma^{\sharp}$
and the bounding halfspace
$$ 
H 
\ := \ 
\{v \in \QQ^n; \; \bangle{u,v} \geq -1\}
\ \subseteq \
\QQ^n
$$
of $A_X$ defined by $u$
satisfies $\sigma^{\sharp} \cap A_X = \sigma^{\sharp} \cap  H$. 
If the ray $\varrho'$ is not contained in $A_{X}$, 
then its leaving point $v'_{\varrho'}$ is the 
intersection point of $\varrho'$ and $\partial H$. 
In this case, we obtain
$$
\bangle{u, v_{\varrho'}} 
\ = \
\frac{\Vert v_{\varrho'} \Vert}{\Vert v'_{\varrho'} \Vert}
\bangle{u, v'_{\varrho'}}
\ = \ 
- \frac{\Vert v_{\varrho'} \Vert}{\Vert v'_{\varrho'} \Vert},
\qquad
\qquad
a_{\varrho'}
\ = \
-1 + 
\frac{\Vert v_{\varrho'} \Vert}{\Vert v'_{\varrho'} \Vert}.
$$
If $\varrho' \subseteq A_X$ holds, then $\varrho'$ is contained in
$H$.
This means 
$\bangle{u, v}  \ge  -1$ for all $v \in \varrho'$.
It follows $\bangle{u, v_{\varrho'}} \geq 0$ 
and thus $a_{\varrho'} \le -1$. 
\end{proof}

\begin{proof}[Proof of Theorem~\ref{thm:main}, 
Assertions~(ii) to~(v)]
We prove the ``if'' parts first; recall that for them 
we only require the existence of one tropical resolution. 
Let $\varphi \colon X' \to X$ be such a resolution, given 
by a subdivision of fan $\Sigma' \to \Sigma$. 
Given a ray $\varrho \in \Sigma'$ not belonging to 
$\Sigma$, we have to show that the discrepancy 
$a_\varrho$ satisfies the desired bounds.
For~(ii), let $A_X^c$ be bounded. 
Then $\varrho \nsubseteq A_X^c$ and
Proposition~\ref{prop:discrep} gives 
$a_\varrho > -1$.
In assertions~(iii) to~(v) observe 
that $A_X^c$ is bounded and thus 
$\varrho \nsubseteq A_X^c$. 
In~(iii), the intersection point of 
$\varrho$ and $\partial \varepsilon A_X^c$ 
is $\varepsilon v'_\varrho$. 
By assumption, $v_\varrho \notin \varepsilon A_X^c$.
Thus thus $\Vert v_\varrho\Vert > \varepsilon \Vert v'_\varrho \Vert$
and Proposition~\ref{prop:discrep} gives 
$a_\varrho > -1 + \varepsilon$.
Similarly, for~(iv) and (v), the intersection point 
of $\varrho$ and $\partial A_X^c$ is $v'_\varrho$,
and Proposition~\ref{prop:discrep} gives 
$a_\varrho \ge 0$ in~(iv) and $a_\varrho > 0$ in~(v).

We turn to the ``only if'' parts. Here we required that 
$X$ is strongly tropically resolvable. 
For~(ii), assume that $A_X^c$ is not bounded.
Then $A_X^c$ contains a ray $\varrho$. 
Let $X'\rightarrow X$ be a tropical resolution 
with $\varrho \in \Sigma'$. 
Then $a_\varrho \leq -1$ by Proposition~\ref{prop:discrep},
a contradiction. 
In Assertions~(iii) to~(v), $A_X^c$ is bounded due 
to~(ii).
For~(iii), assume that $\varepsilon A_X^c$ contains 
an integral point $v \neq 0$
and set $\varrho := \cone(v)$.  
Let $X'\rightarrow X$ a tropical resolution with 
$\varrho \in \Sigma'$. 
Then $\varrho$ and $\varepsilon \partial A_X^c$ intersect 
at $\varepsilon v'_\varrho$. 
Because $v_\varrho \in \varepsilon A_X^c$,
we have $\Vert v_\varrho \Vert \le \varepsilon \Vert v'_\varrho \Vert$
and Proposition~\ref{prop:discrep} gives 
$a_\varrho \le -1 + \varepsilon$, a contradiction. 
Similarly, for~(iv) and~(v), assume that $A_X^c$ 
contains an (inner) integral point $v \neq 0$ generating 
a ray $\varrho$ that does not belong to $\Sigma$. 
Let $X'\rightarrow X$ a tropical resolution with $\varrho \in \Sigma'$. 
Proposition~\ref{prop:discrep} implies $a_\varrho < 0$ in case~(iv)
and $a_\varrho \le 0$ in case~(v), a contradiction. 
\end{proof}

\begin{remark}
The assignment $\eta \mapsto \cone(\eta)$ defines an 
order-preserving bijection between the anticanonical 
complex $A_X^c$ and the fan $\Sigma \sqcap \trop(X)$.
\end{remark}

We conclude the section with some observations that may be drawn 
for the intersection of $A_X$ with the tropical lineality space.

\begin{definition}
Let $\trop_0(X) \subseteq \trop(X)$ 
denote the lineality space of the 
tropical variety.
The \emph{lineality part} of the anticanonical
complex is the polyhedral complex 
$A_{X,0}^c := A_X \sqcap \trop_0(X)$.
\end{definition}

\begin{proposition}
Let $X$ be a log terminal Fano variety
and let $\vert A_{X,0}^c \vert$ denote 
the support of the lineality part
of the anticanonical complex $A_X^c$.
\begin{enumerate}
\item
$\vert A_{X,0}^c \vert$ is a full dimensional  
polytope in $\trop_0(X)$ having the origin 
as an interior point.
\item
If $X$ is $\varepsilon$-log terminal
then the origin is the only lattice 
point of $\varepsilon \vert A_{X,0}^c \vert$.
\item
If $X$ is canonical then the origin is 
the only interior lattice 
point of $\vert A_{X,0}^c \vert$.
\item
If $X$ is  terminal
then the origin is the only lattice 
point of $\vert A_{X,0}^c \vert$.
\end{enumerate}
\end{proposition}


\section{Fano varieties with torus action of complexity one}
\label{sec:fanocpl1}

We take a closer look at Fano varieties with a torus 
action of complexity one. 
First we recall the approach to rational varieties with 
torus action of complexity one provided by~\cite{HaSu,HaHe}.
The Cox rings of these varieties are precisely the rings 
obtained in the following way.

\begin{construction}
\label{constr:RAPdown}
Fix $r \in \ZZ_{\ge 1}$, a sequence 
$n_0, \ldots, n_r \in \ZZ_{\ge 1}$, set 
$n := n_0 + \ldots + n_r$, and fix  
integers $m \in \ZZ_{\ge 0}$ and $0 < s < n+m-r$.
The input data are 
\begin{itemize}
\item 
a matrix $A := [a_0, \ldots, a_r]$ 
with pairwise linearly independent 
column vectors $a_0, \ldots, a_r \in \KK^2$,
\item 
an integral block matrix $P$ of size 
$(r + s) \times (n + m)$, the columns 
of which are pairwise different primitive
vectors generating $\QQ^{r+s}$ as a cone.
\begin{eqnarray*}
P
& = & 
\left[ 
\begin{array}{cc}
L & 0 
\\
d & d'  
\end{array}
\right],
\end{eqnarray*}
where $d$ is an $(s \times n)$-matrix, $d'$ an $(s \times m)$-matrix 
and $L$ an $(r \times n)$-matrix built from tuples 
$l_i := (l_{i1}, \ldots, l_{in_i}) \in \ZZ_{\ge 1}^{n_i}$ 
as follows
\begin{eqnarray*}
L
& = & 
\left[
\begin{array}{cccc}
-l_0 & l_1 &   \ldots & 0 
\\
\vdots & \vdots   & \ddots & \vdots
\\
-l_0 & 0 &\ldots  & l_{r} 
\end{array}
\right].
\end{eqnarray*}
\end{itemize}
Consider the polynomial ring 
$\KK[T_{ij},S_k]$ in the variables 
$T_{ij}$, where 
$0 \le i \le r$, $1 \le j \le n_i$,
and $S_k$, where $1 \le k \le m$.
For every $0 \le i \le r$, define a monomial
\begin{eqnarray*}
T_i^{l_i} 
&  := &
T_{i1}^{l_{i1}} \cdots T_{in_i}^{l_{in_i}}.
\end{eqnarray*}
Denote by $\mathfrak{I}$ the set of 
all triples $I = (i_1,i_2,i_3)$ with 
$0 \le i_1 < i_2 < i_3 \le r$ 
and define for any $I \in \mathfrak{I}$ 
a trinomial 
$$
g_I
\ := \
g_{i_1,i_2,i_3}
\ := \
\det
\left[
\begin{array}{ccc}
T_{i_1}^{l_{i_1}} & T_{i_2}^{l_{i_2}} & T_{i_3}^{l_{i_3}}
\\
a_{i_1} & a_{i_2} & a_{i_3}
\end{array}
\right].
$$
Let $P^*$ denote the transpose of $P$,
consider the factor group 
$K := \ZZ^{n+m}/\rm{im}(P^*)$
and the projection $Q \colon \ZZ^{n+m} \to K$.
We define a $K$-grading on 
$\KK[T_{ij},S_k]$ by setting
$$ 
\deg(T_{ij}) 
 \ := \ 
Q(e_{ij}),
\qquad
\deg(S_{k}) 
 \ := \ 
Q(e_{k}).
$$
Then the trinomials $g_I$ just introduced 
are $K$-homogeneous, all of the same degree.
In particular, we obtain a $K$-graded 
factor ring  
\begin{eqnarray*}
R(A,P)
& := &
\KK[T_{ij},S_k; \; 0 \le i \le r, \, 1 \le j \le n_i, 1 \le k \le m] 
\ / \
\bangle{g_I; \; I \in \mathfrak{I}}.
\end{eqnarray*}
\end{construction}

\begin{remark}
\label{rem:ci}
The $K$-graded ring $R(A,P)$ of Construction~\ref{constr:RAPdown}
is a complete intersection: with $g_i := g_{i,i+1,i+2}$ 
we have 
$$ 
\bangle{g_I; \; I \in \mathfrak{I}}
\ = \ 
\bangle{g_0,\ldots,g_{r-2}},
\qquad\quad
\dim(R(A,P)) \ = \ n+m-(r-1).
$$
We can always assume that $P$ is \emph{irredundant} 
in the sense that $l_{i1} + \ldots + l_{in_i} \ge 2$ holds 
for $i = 0,\ldots, r$; note that a redundant $P$ allows 
the elimination of variables in $R(A,P)$.
\end{remark} 

\begin{remark}
\label{rem:fanoRAP}
The \emph{anticanonical class} of the $K$-graded 
ring $R(A,P)$ from Construction~\ref{constr:RAPdown}
is 
$$ 
\kappa(A,P)
\ := \ 
\sum_{i,j} Q(e_{ij})
\ + \ 
\sum_k Q(e_k)
\ - \ 
(r-1) \sum_{j=0}^{n_0} l_{0j} Q(e_{0j})
\ \in \ 
K
$$
and the \emph{moving cone} of $R(A,P)$ in $K_\QQ$ is
$$ 
\Mov(A,P) 
\ := \ 
\bigcap_{i,j} \cone(Q(e_{uv},e_{t}; \; (u,v) \ne (i,j))  
\ \cap \ 
\bigcap_{k} \cone(Q(e_{uv},e_{t}; \; t \ne k).
$$
The $K$-graded ring $R(A,P)$ is the Cox ring of 
a Fano variety if and only if $\kappa(A,P)$ belongs
to the relative interior of $\Mov(A,P)$. 
\end{remark}

\begin{construction}
\label{constr:RAPFano}
Consider the $K$-graded ring $R(A,P)$ 
of Construction~\ref{constr:RAPdown}
and assume that $\kappa(A,P)$ lies in the 
relative interior of $\Mov(A,P)$.
Then the $K$-grading on $\KK[T_{ij},S_k]$
defines an action of the quasitorus 
$H := \Spec \; \KK[K]$ on $\b{Z} := \KK^{n+m}$
leaving $\b{X} := V(g_I; \; I \in \mathfrak{I}) \subseteq \b{Z}$ 
invariant. Consider 
$$ 
\rq{Z}_c 
\ := \ 
\{z \in \b{Z}; \; 
f(z) \ne 0 \text{ for some } 
f \in \KK[T_{ij},S_k]_{ \nu \kappa(A,P)}, \, 
\nu \in \ZZ_{>0} \}
\ \subseteq \ 
\b{Z},
$$
the set of $H$-semistable points with 
respect to the weight $\kappa(A,P)$.
Then $\rq{X} := \b{X} \cap \rq{Z}_c$ is 
an open $H$-invariant set in $\b{X}$ 
and we have a commutative diagram
$$ 
\xymatrix{
{\rq{X}}
\ar[r]
\ar[d]_{\quot H}
&
{\rq{Z}_c}
\ar[d]^{\quot H}
\\
X(A,P)
\ar[r]
&
Z_c}
$$
where $X(A,P)$ is a Fano variety with torus action 
of complexity one,
$Z_c := \rq{Z}_c \quot H$ is a toric Fano variety,
the downward maps are characteristic spaces and
the lower horizontal arrow is a closed embedding. 
We have
$$ 
\dim(X(A,P)) = s+1,
\qquad
\Cl(X(A,P)) \ \cong \ K,
$$
$$
-\mathcal{K}_X \ = \ \kappa(A,P),
\qquad
\mathcal{R}(X) \ \cong \ R(A,P).
$$
\end{construction}

By the results of~\cite{HaSu,HaHe} every normal rational 
Fano variety with a torus action of complexity one arises 
from this construction. 

\begin{remark}
\label{remark:admissibleops}
The following elementary column and row operations 
on the defining matrix $P$ do not change 
the isomorphy type of the associated Fano variety $X(A,P)$;
we call them \emph{admissible operations}:
\begin{enumerate}
\item
swap two columns inside a block
$v_{ij_1}, \ldots, v_{ij_{n_i}}$,
\item
swap two whole column blocks
$v_{ij_1}, \ldots, v_{ij_{n_i}}$
and $v_{i'j_1}, \ldots, v_{i'j_{n_{i'}}}$,
\item
add multiples of the upper $r$ rows
to one of the last $s$ rows,
\item
any elementary row operation among the last $s$
rows,
\item
swap two columns inside the $d'$ block.
\end{enumerate}
The operations of type (iii) and (iv) do not change
the associated ring $R(A,P)$, whereas the 
types (i), (ii), (v)
correspond to certain renumberings of the variables
of $R(A,P)$ keeping the (graded) isomorphy type.
\end{remark}

We now discuss the resolution of singularities in this 
setting. 
The references for complete proofs 
are~\cite[Sec.~3.4.4]{ArDeHaLa} and~\cite{Hug}.
A local version of our desingularization using another 
approach was given in~\cite{LiSu}.

\begin{construction}
\label{rem:cpl1isstr}
Consider the setting of Construction~\ref{constr:RAPFano}.
Let $\gamma \subseteq \QQ^{n+m}$ be the positive orthant
and for a face $\gamma_0 \preceq \gamma$ let 
$\gamma_0^* := \gamma_0^\perp \cap \gamma$ be the complementary
face. Then the fan of $Z_c = {\rq{Z}_c \quot H}$ is
$$ 
\Sigma_c
\ = \ 
\{P(\gamma_0^*); \; 
\gamma_0 \preceq \gamma, \
\kappa(A,P) \in \relint(Q(\gamma_0))\}.
$$
In particular, the primitive generators of the rays of $\Sigma$ 
are precisely the columns $v_{ij}$ and $v_k$ of the matrix $P$.
With $P_0 = [L,0]$ and $P_1 = [E_r,0]$, 
where $E_r$ is the $r \times r$ unit matrix,
we have a commutative diagram
$$
\xymatrix{
{\ZZ^{n+m}}
\ar[rr]^{P}
\ar[dr]_{P_0}
&&
{\ZZ^{r+s}}
\ar[dl]^{P_1}
\\
&
{\ZZ^r}
&
}
$$
The torus $T$ acting on $X(A,P)$ is the subtorus $T \subseteq \TT^{r+s}$
corresponding to the sublattice 
$\ker(P_1) = 0 \times \ZZ^s \subseteq \ZZ^{r+s}$.
Now, let $e_1,\ldots,e_r \in \ZZ^r$ be the canonical basis
vectors, set
$$ 
\varrho_0
\ := \ 
\cone(-e_1 - \ldots -e_r),
\qquad
\varrho_i 
\ := \ 
\cone(e_i),
\quad
1 \le i \le r,
$$
and consider the fan 
$\Delta(r) := \{0,\varrho_0, \ldots, \varrho_r\}$
in $\ZZ^r$.
Note that $P_1$ sends an $ij$-th column $v_{ij}$ 
of $P$ into the ray $\varrho_i$ and the 
columns $v_k$ to zero.
The tropical variety of $X \cap \TT^n \subseteq Z_c$
is then given as 
$$
\trop(X) 
\ = \ 
\bigcup_{i=0}^r P_1^{-1}(\varrho_i)
\ \subseteq \
\QQ^{r+s}.
$$
The minimal toric ambient variety $Z_\Sigma \subseteq Z_c$ 
of $X \subseteq Z_c$ is the open toric subvariety 
having as closed orbits the minimal orbits of $Z_c$
intersecting $X$.
The fan $\Sigma$ of $Z_\Sigma$ is generated by the cones
of $\Sigma_c$ with $\relint(\sigma) \cap \trop(X) \ne \emptyset$.
Set
$$
\Sigma'
\ := \
\Sigma \sqcap \trop(X)
\ = \
\{
\sigma \cap P_1^{-1}(\varrho_i); \; \sigma \in \Sigma, \, 0 \le i \le r
\}.
$$
Then we have a map of fans $\Sigma' \to \Sigma$
and  the associated birational toric morphism 
$Z_{\Sigma'} \to Z_\Sigma$ fits into a commutative diagram
$$ 
\xymatrix{
X' 
\ar[r]
\ar[d]
&
Z_{\Sigma'}
\ar[d]
\\
X \ar[r]
&
Z_\Sigma
}
$$
where  $X' \subseteq Z_{\Sigma'}$ 
is the proper transform, i.e.,~the 
closure of $X \cap \TT^{r+s}$ in $Z_{\Sigma'}$.
Any regular subdivision $\Sigma'' \to \Sigma'$ 
provides a toric resolution $Z_{\Sigma''} \to Z_{\Sigma'}$ 
and induces a resolution $X'' \to X'$.
\end{construction}

The resulting varieties $X'$ and $X''$ arising 
in this construction are again normal rational 
varieties with torus action of complexity one
and have Cox rings of the form $R(A,P')$ and 
$R(A,P'')$ as presented in Construction~\ref{constr:RAPdown},
see~\cite[Thm.~3.4.4.9]{ArDeHaLa}.
In particular they are Mori dream spaces and 
we obtain the following.

\begin{proposition}
\label{prop:cpl1isstr}
Let $X = X(A,P)$ be a Fano variety as in 
Construction~\ref{constr:RAPFano}.
Then $X$ is strongly tropically resolvable.
\end{proposition}


\section{Structure of the anticanonical complex}
\label{sec:cpl1-antican}

The notation is the same as in Section~\ref{sec:fanocpl1}.
We consider a $\QQ$-factorial rational Fano 
variety~$X = X(A,P)$ with torus action of complexity 
one and investigate the structure of the associated 
anticanonical complex $A_X^c$.
Combining the results with Theorem~\ref{thm:main},
we derive first bounding conditions on the entries 
of the defining matrix $P$.

Recall that we have $X \subseteq Z_\Sigma \subseteq Z_c$, 
where $Z_c$ is a toric Fano variety and $Z_\Sigma$ is the 
minimal open toric subvariety of $Z_c$ containing~$X$
as a closed subvariety.
The fans $\Sigma_c$ of $Z_c$ and $\Sigma$ of $Z_\Sigma$ 
live in the lattice $\ZZ^{r+s}$.
They share the same set of rays $\varrho$
and the primitive generators 
$v_{\varrho} \in \varrho$ are precisely the columns
of the matrix
$$ 
P 
\ = \ 
\left(
\begin{array}{rrrrr}
-l_0   & l_1 &        & 0 & 0
\\
\vdots &     &  \ddots  & & \vdots
\\
-l_0   &  0  &          & l_r & 0
\\
d_0    & d_1 &          & d_r & d'
\end{array}
\right),
$$
where each $l_i = (l_{i1}, \ldots, l_{in_i})$ is 
an $1 \times n_i$ block and each
$d_i = (d_{i1}, \ldots, d_{in_i})$ is an $s \times n_i$ 
block.
The tropical variety $\trop(X)$ with its quasifan structure 
also lives in $\ZZ^{r+s}$.
With $\lambda := 0 \times \QQ^s \subseteq \QQ^{r+s}$, 
the canonical basis vectors $e_1,\ldots, e_r$ and
$e_0 := -e_1- \ldots -e_r$, we have
$$ 
\trop(X) 
\ = \ 
\tau_0 \cup \ldots \cup \tau_r
\ \subseteq \ 
\QQ^{r+s},
\qquad \text{where} \quad
\tau_i \ := \ \cone(e_i) + \lambda.
$$
Note that this defines the coarsest possible quasifan structure
on $\trop(X)$, and the lineality space of this quasifan is $\lambda$.

\begin{definition}
A cone $\sigma \in \Sigma$ is called
\emph{big}, if $\sigma \cap \relint(\tau_i) \ne \emptyset$ 
holds for each $i = 0, \ldots, r$.
An \emph{elementary big} cone is a big cone $\sigma \in \Sigma$ 
having no rays inside $\lambda$ and precisely one inside 
$\tau_i$ for each $i = 0, \ldots, r$.
A \emph{leaf cone} is a $\sigma \in \Sigma$ such that
$\sigma \subseteq \tau_i$ holds for some $i$. 
\end{definition}

\begin{remark}
The big cones and the leaf cones are precisely those cones 
$\sigma \in \Sigma$ such that $\relint(\sigma)$ intersects
$\trop(X)$. 
The latter property, by Tevelev's criterion~\cite[Lemma~2.2]{Tev},
merely means that
the big cones and the leaf cones describe
precisely the toric orbits of $Z$ intersecting $X$.
Observe that all maximal cones of $\Sigma$
are big cones or leaf cones.
\end{remark}

\begin{definition}
\label{def:ellsigma}
Let $\sigma \in \Sigma$ be an elementary big cone.
We assign the following integers
to the rays $\varrho = \cone(v_{ij}) \in \sigma^{(1)}$
of $\sigma$ and to $\sigma$ itself:
$$
l_\varrho \ := \ l_{ij},
\qquad
\ell_{\sigma,\varrho} 
\ := \ 
l_{{\varrho}}^{-1} \prod_{\varrho' \in \sigma^{(1)}}{l_{\varrho'}},
\qquad 
\ell_{\sigma} \ 
:= \ 
\sum_{\varrho \in \sigma^{(1)}}{\ell_{\sigma,\varrho}} 
- 
(r-1)\prod_{\varrho \in \sigma^{(1)}}{l_{\varrho}}.
$$
Moreover, in $\QQ^{r+s}$, we define vectors and a ray:
$$
v_{\sigma} 
\ := \ 
\sum_{\varrho \in \sigma^{(1)}}{\ell_{\sigma,\varrho} v_{\varrho}},
\qquad\qquad
v'_{\sigma}
\ := \ 
\ell_{\sigma}^{-1}v_{\sigma},
\qquad\qquad
\varrho_\sigma 
\ := \ 
\cone(v_{\sigma}).
$$
Finally, we denote by $c_\sigma$ the greatest common divisor 
of the entries of the vector $v_\sigma \in \ZZ^{r+s}$.
\end{definition}

The first structural statement describes the rays of the 
coarsest common refinement $\Sigma \sqcap \trop(X)$ of the 
fan $\Sigma$ and the tropical variety $\trop(X)$ regarded 
as a quasifan.

\begin{proposition}
Let $X = X(A,P)$ be a $\QQ$-factorial Fano variety. 
\begin{enumerate}
\item 
For every elementary big cone $\sigma \in \Sigma$,
we have $\sigma \cap \lambda = \varrho_\sigma$; in 
particular, $\varrho_\sigma$ lies in the lineality space 
$\lambda$.
\item
The set of rays of $\Sigma \sqcap \trop(X)$ consists 
of the rays $\varrho \in \Sigma$ and the rays 
 $\varrho_\sigma$, where $\sigma \in \Sigma$ runs through 
the elementary big cones.
\end{enumerate}
\end{proposition}

\begin{proof}
For~(i), one directly computes the intersection
$\sigma \cap \lambda$.
We prove~(ii). 
Since all rays of $\Sigma$ lie on $\trop(X)$,
the rays of $\Sigma$ are also rays of 
$\Sigma \sqcap \trop(X)$.
By~(i), the $\varrho_{\sigma}$, where $\sigma \in \Sigma$ 
is elementary big, are rays of $\Sigma \sqcap \trop(X)$.
Let $\varrho' \in \Sigma \sqcap \trop(X)$ be any ray not 
belonging to $\Sigma$.
Then there exist cones $\sigma \in \Sigma$ and $\tau \in \trop(X)$
which satisfy $\sigma \cap \tau = \varrho'$ 
and which are minimal 
with this property.
The latter means that $\relint(\varrho') = 
\relint(\sigma) \cap \relint(\tau)$.

To obtain $\tau = \lambda$, we have to exclude the 
case $\tau = \tau_i$ for some $i = 0,\ldots,r$.
Indeed if $\tau = \tau_i$, then no ray $\varrho \preceq \sigma$ 
lies in $\tau_i$, because otherwise we had
$\varrho \subseteq \sigma \cap \tau_i = \varrho'$,
contradicting $\varrho' \not \in \Sigma$.
Thus, $\sigma$ has no rays inside $\tau_i$.
Since all rays of $\sigma$ lie on $\trop(X)$,
we conclude $\relint(\sigma) \cap \relint(\tau_i) = \emptyset$,
a contradiction.

We show that $\sigma$ is an elementary big cone.
Firstly, $\sigma$ must be big because otherwise 
we had $\relint(\sigma) \cap \lambda = \emptyset$.
Since $X$ is $\QQ$-factorial, $\sigma$ is simplicial.
Thus there exists 
an elementary big face $\eta$ of $\sigma$. 
But then $\varrho_{\eta} = \eta
\cap \lambda \preceq \sigma \cap \lambda = \varrho'$ 
which implies $\varrho' = \varrho_{\eta}$.
By minimality of $\sigma$, we conclude $\sigma = \eta$.
\end{proof}

We  take a closer look at the discrepancies 
of a tropical resolution of singularities 
along the divisors corresponding to the rays 
$\varrho_\sigma$.

\begin{proposition}
Let $X = X(A,P)$ be a $\QQ$-factorial Fano variety
and $\sigma \in \Sigma$ an elementary big cone. 
\begin{enumerate}
\item 
If $\varrho_{\sigma}$ leaves $A_X$, e.g.~if $\sigma$ 
defines a log terminal singularity, then its leaving point is 
$v'_{\varrho_{\sigma}} = \ell_{\sigma}^{-1}v_{\sigma} = v'_{\sigma}$.
\item
For any tropical resolution $\varphi \colon X' \to X$
of singularities,
the discrepancy along the divisor corresponding to
$\varrho_{\sigma}$ is 
$a_{\varrho_{\sigma}} = -1 + c_{\sigma}^{-1}\ell_{\sigma}$. 
\end{enumerate}
\end{proposition}

\begin{proof}
Recall that the intersection point $v_{\varrho_\sigma}'$ 
of the ray $\varrho_\sigma$ 
with the boundary $\partial A_X^c$ is defined by
$$ 
\bangle{u, v_{\varrho_\sigma}'}
\ = \ 
-1,
\qquad
\text{where}
\quad
u \ := \ 
(P^*)^{-1}(e_{-\mathcal{K}_X} + e - e_{\Sigma})
$$
with any vertex $e_{-\mathcal{K}_X} + e - e_{\Sigma}$ 
of $B(-\mathcal{K}_X) + B - e_{\Sigma}$ minimizing  
$\rq{v}_{\sigma} := \sum_{\varrho \in \sigma^{(1)}} \ell_{\sigma, \varrho} e_{\varrho}$.
For $v_\sigma = P(\rq{v}_{\sigma})$, we obtain 
$$
\bangle{u, v_{\sigma}} 
\  = \ 
\bangle{e_{-\mathcal{K}_X}, \rq{v}_{\sigma}}
+ 
\bangle{e,\rq{v}_{\sigma}} 
- 
\bangle{e_{\Sigma}, \rq{v}_{\sigma}} 
\  = \ 
\bangle{e,\rq{v}_{\sigma}} 
- 
\bangle{e_{\Sigma}, \rq{v}_{\sigma}}.
$$
To compute further, set
$\rq{u}_i := \sum_{\varrho \in \Rho_i} l_{\varrho} e_{\varrho}$ for $i=0,\ldots,r$,
where $\Rho_i$ denotes the set of rays of $\Sigma$ contained in $\tau_i$. 
Denoting by $\varrho_i$ the unique ray of $\sigma$ in $\tau_i$, we have
$$
\bangle{\rq{u}_i, \rq{v}_{\sigma}}
\ = \ 
l_{\varrho_i} \ell_{\sigma, \varrho_i} 
\ = \ 
\prod_{\varrho \in \sigma^{(1)}} l_{\varrho}.
$$
Consequently, for any point $e \in B = B(g_0) + \ldots + B(g_{r-2})$, 
we obtain
$$
\bangle{e,\rq{v}_{\sigma}} 
\ = \ 
(r-1) \prod_{\varrho \in \sigma^{(1)}} l_{\varrho}.
$$
Thus, we obtain $\bangle{u, v_{\sigma}}  = -\ell_{\sigma}$
and the leaving point is  
$v_{\varrho_\sigma}' = \ell_{\sigma}^{-1}v_{\sigma} = v'_{\sigma}$
as claimed in~(i).
Assertion~(ii) is then a direct application of 
Proposition~\ref{prop:discrep}.
\end{proof}

As an application, we obtain first bounding conditions 
on the entries $l_\varrho$ of the defining matrix $P$ 
in terms of the singularities of $X$.

\begin{corollary}
\label{cor:sing2lbound}
Let $X = X(A,P)$ be a $\QQ$-factorial Fano variety
and $\sigma \in \Sigma$ an elementary big cone. 
If the singularity defined by $\sigma$ is 
\begin{enumerate}
\item 
log terminal, then 
$\sum_{\varrho \in \sigma^{(1)}}{l_{\varrho}^{-1}} > r-1$,
\vspace{2pt}
\item 
$\varepsilon$-log terminal, then 
$\sum_{\varrho \in \sigma^{(1)}}{l_{\varrho}^{-1}} 
> r-1 + \varepsilon c_{\sigma}\prod_{\varrho \in \sigma^{(1)}}{l_{\varrho}^{-1}}$,
\vspace{2pt}
\item 
canonical, then 
$\sum_{\varrho \in \sigma^{(1)}}{l_{\varrho}^{-1}} \geq 
r-1 + c_{\sigma}\prod_{\varrho \in \sigma^{(1)}}{l_{\varrho}^{-1}}$,
\vspace{2pt}
\item 
terminal, then
$\sum_{\varrho \in \sigma^{(1)}}{l_{\varrho}^{-1}} > 
 r-1 + c_{\sigma}\prod_{\varrho \in \sigma^{(1)}}{l_{\varrho}^{-1}}$.
\end{enumerate}
\end{corollary}

\begin{corollary}
\label{cor:logterm2lbound}
Let $X = X(A,P)$ be a $\QQ$-factorial Fano variety
and consider an elementary big cone
$\sigma = \varrho_0 + \ldots + \varrho_r \in \Sigma$
defining a log terminal singularity. 
Assume $l_{\varrho_0} \ge  \ldots \ge l_{\varrho_r}$.  
Then $l_{\varrho_3} =  \ldots = l_{\varrho_r} = 1$
holds and $(l_{\varrho_0},l_{\varrho_1},l_{\varrho_2})$ 
is a platonic triple, i.e.,~one of
$$ 
(l_{\varrho_0},l_{\varrho_1},1),
\qquad
(l_{\varrho_0},2,2),
\qquad
(3,3,2),
\qquad
(4,3,2),
\qquad
(5,3,2).
$$
According to these possibilities, the number $\ell_{\sigma}$ is given as
\begin{align*}
  \ell_{\sigma} 
 &\ = \
 l_{\varrho_0} l_{\varrho_1} + 
 l_{\varrho_0} l_{\varrho_2} + 
 l_{\varrho_1} l_{\varrho_2} - 
 l_{\varrho_0} l_{\varrho_1} l_{\varrho_2} 
 \\
&\ = \ 
\begin{cases}
l_{\varrho_0}+l_{\varrho_1}, & \text{if }
(l_{\varrho_0},l_{\varrho_1},l_{\varrho_2})
=
(l_{\varrho_0},l_{\varrho_1},1),
\\
4, & \text{if }
(l_{\varrho_0},l_{\varrho_1},l_{\varrho_2}),
=
(l_{\varrho_0},2,2),
\\ 
3, & \text{if }
(l_{\varrho_0},l_{\varrho_1},l_{\varrho_2})
= (3,3,2),
\\
2, & \text{if }
(l_{\varrho_0},l_{\varrho_1},l_{\varrho_2})
=(4,3,2),
\\
1, & \text{if }
(l_{\varrho_0},l_{\varrho_1},l_{\varrho_2})
= (5,3,2).
\end{cases}
\end{align*}
\end{corollary}

\begin{corollary}
\label{cor:boundrel}
Let $X = X(A,P)$ be a log terminal $\QQ$-factorial
Fano variety.
Assume that $P$ is irredundant and
$\Sigma$ contains a big cone. 
Then the number $r-1$ of relations is bounded~by 
$$
r-1
\ \leq \ 
\dim(X) + \rk(\Pic(X)) .
$$
\end{corollary}

\begin{proof}
Since $X$ is $\QQ$-factorial, $\Pic(X)$ is of rank $n + m - r -s$. 
Let $I \subseteq \{0,\ldots,r\}$ be the set of indices 
with $n_i >1$ and set $n_I := \sum_{i \in I}{n_i}$. 
Then the rank of $\Pic(X)$ equals $n_I + m - |I| - s$. 
Since there exists a big cone, there is also an
elementary big cone 
$\sigma = \varrho_0 + \ldots + \varrho_r \in \Sigma$. 
Since $P$ is irredundant, $l_{\varrho_i} >1$ holds 
for all $i \not\in I$.
Corollary~\ref{cor:logterm2lbound} yields $|I| \geq r-2$. 
We conclude 
$$
\rk(\Pic(X)) 
\ = \
m + n_I - |I| - s 
\ \ge \
2 |I| - |I| - s 
\ \ge \ 
r -2 - s 
\ = \ 
r -1 - \dim(X).
$$
\end{proof}

\begin{definition}
Let $A_X^c$ be the anticanonical complex of $X = X(A,P)$.
Recall that the lineality part of $A_X^c$ is the polyhedral 
complex $A_{X,0}^c = A_X^c \sqcap \lambda$.
The \emph{$i$-th leaf} of $A_X^c$ is the polyhedral complex
$A_X^c \sqcap \tau_i$.
\end{definition}

\begin{corollary}
\label{thm:cpl1antican}
Let $X = X(A,P)$ be a log terminal $\QQ$-factorial 
Fano variety. 
Then the vertices of the anticanonical complex $A_X^c$ 
are precisely the points $v_\varrho$ and $v_\sigma'$, 
where $\varrho$ runs through the rays and $\sigma$ through 
the elementary big cones of $\Sigma$. 
In particular, for the supports of the lineality part 
and the leaves of $A_X^c$, we obtain
\begin{eqnarray*}
\vert A_X^c \sqcap \lambda \vert
& = &
\conv(v_{\varrho}, v_{\sigma}'; \ 
\varrho \in \Sigma \text{ with } \varrho \subseteq \lambda,
\ \sigma \in \Sigma \text { elementary big}),
\\
\vert A_X^c \sqcap \tau_i \vert
& = &
\conv(v_{\varrho}, v_{\sigma}'; \ 
\varrho \in \Sigma \text{ with } \varrho \subseteq \tau_i,
\ \sigma \in \Sigma \text { elementary big}).
\end{eqnarray*} 
\end{corollary}

\begin{remark}
Let $X = X(A,P)$ be a $\QQ$-factorial Fano variety and 
$X'$ the variety arising from the tropical refinement 
$\Sigma \sqcap \trop(X)$. 
Then $A_{X'}^c$ and $A_X^c$ both generate $\Sigma \sqcap \trop(X)$ 
but do not in general coincide, because the rays $\varrho_{\sigma}$ 
of big elementary cones $\sigma \in \Sigma$ intersect the 
boundary of $A_{X'}^c$ in integral points,
whereas the intersection points $v'_{\sigma}$ 
with $A_X^c$ do not need to be integral.
\end{remark}

\begin{remark}
The anticanonical complex $A_X^c$ of a Fano variety $X = X(A,P)$
can also be obtained in the following way.
Since the defining relations $g_l$ of $\mathcal{R}(X)$ 
all have the same $K$-degree, 
we may define $A_X^c$ in a slightly different way 
by exchanging $B(g)$ for 
$$
B(g)' 
\ := \ 
(r-1) \conv(\mu_0,\ldots,\mu_r),
$$
where $\mu_0,\ldots,\mu_r$ are the exponent vectors occuring 
in $g_0,\ldots,g_{r-2}$. 
Then $(r-1) u_l - e_{\Sigma}$ is a representative of $- \mathcal{K}_X$ 
and all the proofs work 
in exactly the same way. On the pro side we note that $B(g)'$ 
does not depend on 
the enumeration of the variables $T_{\varrho}$, while $B(g)$ does. 
\end{remark}


\section{Terminal Fano threefolds}
\label{sec:terminal3folds}

Here we show how to obtain the classification of 
terminal $\QQ$-factorial Fano threefolds~$X$ 
of Picard number one coming with an effective 
action of a two-dimensional torus
given in Theorem~\ref{thm:classif}. 
First recall the following.

\begin{remark}
For any Fano variety $X$ with at most log terminal
singularities, the 
divisor class group $\Cl(X)$ is finitely generated; 
see~\cite[Sec.~2.1]{IsPr}.
If $X$ comes in addition with a torus 
action of complexity one, then $X$ is 
rational and its Cox ring is finitely 
generated;  
see~\cite[Remark~IV.4.1.5]{ArDeHaLa}.
\end{remark}

This allows us to work 
in terms of the defining data $(A,P)$ of $X$
and the notation of Constructions~\ref{constr:RAPdown}
and~\ref{constr:RAPFano},
where we always choose $P$ to be irredundant.
The main step is to derive suitable
effective bounds on the entries of $P$.
According to Theorem~\ref{thm:main}, terminality 
of $X$ is equivalent to the fact that the anticanonical 
complex $A_X^c$ contains no lattice points except 
the origin and the columns of the defining matrix~$P$.
A first observation towards bounds for the shape of $P$
is that log-terminality leads to the following situations.

\begin{lemma}
\label{lem:piconeconfigs}
Let $X = X(A,P)$ a non-toric log terminal
$\QQ$-factorial Fano threefold 
of Picard number one, where $P$ is irredundant. 
Then, after suitable admissible operations, 
$P$ suits into one of the following cases:
\begin{enumerate}
\item
 $m=0$, $r=2$ and $n=5$, where $n_0 = n_1 = 2$, $n_2 = 1$.
\item
 $m=0$, $r=3$ and $n=6$, where $n_0 = n_1 = 2$, $n_2 = n_3 = 1$.
\item
 $m=0$, $r=4$ and $n=7$, where $n_0 = n_1 = 2$, $n_2 = n_3 = n_4 = 1$.
\item
 $m=0$, $r=2$ and $n=5$, where $n_0 = 3$, $n_1 = n_2 = 1$.
\item
 $m=0$, $r=3$ and $n=6$, where $n_0 = 3$, $n_1 = n_2 = n_3 = 1$.
\item
$m=1$, $r=2$ and $n=4$, where  $n_0 = 2$, $n_1 = n_2 = 1$.
\item
$m=1$, $r=3$ and $n=5$, where  $n_0 = 2$, $n_1 = n_2 = n_3 = 1$.
\item
$m=2$, $r=2$ and $n=3$, where  $n_0 = n_1 = n_2 = 1$.
\end{enumerate}
\end{lemma}

\begin{proof}
Since $X$ is non-toric, there is at least one relation 
in the Cox ring. This implies $r\geq 2$.
Since $X$ is of Picard number one, there is an elementary 
big cone and thus Corollary~\ref{cor:boundrel} yields
$r\leq 5$. Using $n+m = \dim(X) + r$, we obtain
$$ 
2 \ \le \ r \ \le \ 5, 
\qquad
r + 1 \ \le \ n,
\qquad
n+m  = r + 3.
$$
Combining Corollary~\ref{cor:logterm2lbound} with
the fact that $P$ is irredundant, 
we see that at most three of the $n_i$ equal one. 
This leaves us with the cases listed in the assertion.
\end{proof}

We treat exemplarily Situation~(i) of 
Lemma~\ref{lem:piconeconfigs}.
This case reflects all the occurring arguments.
The final bounds on the defining matrix $P$ 
are given in Propositions~\ref{prop:l11-l12-1} 
to~\ref{prop:l21-ge3}.
For a treatment of the other situations,
see~\cite[Section~2.4]{Ni}.

\begin{proposition}
\label{prop:higherplt}
Let $X = X(A,P)$ a non-toric terminal
$\QQ$-factorial Fano threefold 
of Picard number one such that $P$ is irredundant
and we have $r=2$, $m=0$ and $n=5$, 
where $n_0 = n_1 = 2$, $n_2 = 1$.
Then $l_{01} = l_{02} = 1$ or $l_{11} = l_{12} = 1$ hold.
\end{proposition}

\begin{proof}
Since $P$ is irredundant, we have $l_{21} \ge 2$.
Moreover, by suitable admissible operations, we achieve
$l_{01} \ge l_{11} \ge l_{12}$, 
$l_{01} \ge l_{02}$.
In total, $P$ is of the form
$$ 
P 
\ = \ 
\left[
\begin{array}{rrrrr}
-l_{01} & -l_{02} & l_{11} & l_{12} & 0
\\
 -l_{01} & -l_{02} & 0 & 0 & l_{21} 
\\
d_{101} & d_{102} & d_{111} & d_{112} & d_{121}
\\
d_{201} & d_{202} & d_{211} & d_{212} & d_{221} 
\end{array}
\right].
$$
We have to show that in the case $l_{11} > 1$,
no terminal $X = X(A,P)$ is left.
According to Corollary~\ref{cor:logterm2lbound}, 
this means to treat the following configurations 
of the~$l_{ij}$:
\begin{center}
\begin{tabular}{r|r|r}
$(l_{01},l_{11},l_{21})$ & $l_{02}$ & $l_{12}$ 
\\
\hline
$(l_{01},2,2)$ & $\le l_{01}$ & $\le 2$ 
\\
\hline
$(2,2,l_{21})$ & $\le 2$ & $\le 2$ 
\\
\hline
$(3,3,2)$ & $\le 3$ & $\le 3$ 
\\
\hline
$(3,2,3)$ & $\le 3$ & $\le 2$ 
\\
\hline
$(4,3,2)$ & $\le 4$ & $\le 3$ 
\\
\hline
$(4,2,3)$ & $\le 4$ & $\le 2$ 
\\
\hline
$(3,2,4)$ & $\le 3$ & $\le 2$ 
\\
\hline
$(5,3,2)$ & $\le 5$ & $\le 3$ 
\\
\hline
$(5,2,3)$ & $\le 5$ & $\le 2$ 
\\
\hline
$(3,2,5)$ & $\le 3$ & $\le 2$ 
\end{tabular}
\end{center}
We first consider the linearity part $A_{X,0}^c$ of 
the anticanonical complex $A_{X}^c$.
Corollary~\ref{thm:cpl1antican} allows an explicit computation.
For the vertex $u$ of $A_{X,0}^c$ defined by the cone 
$\sigma$ corresponding to the platonic triples from the left 
column of the table above we obtain coordinates
$$
u 
\ = \
\Bigl( 
0,0,
\frac{l_{01}l_{11}d_{121}+l_{01}l_{21}d_{111}+l_{11}l_{21}d_{101}}%
{l_{01}l_{11}+l_{01}l_{21}+l_{11}l_{21}-l_{01}l_{11}l_{21}},
\frac{l_{01}l_{11}d_{221}+l_{01}l_{21}d_{211}+l_{11}l_{21}d_{201}}%
{l_{01}l_{11}+l_{01}l_{21}+l_{11}l_{21}-l_{01}l_{11}l_{21}} 
\Bigr).
$$
The (common) denominator of these coordinates is
the $\ell_\sigma$ from Corollary~\ref{cor:logterm2lbound}.
For triples of type $(5,3,2)$ we have $\ell_\sigma=1$
and thus $u$ is integral.
For triples of type $(4,3,2)$ we have $\ell_\sigma=2$
and the numerators are even,
because every summand is a multiple of $2$ or $4$. 
Thus, $u$ is integral again.
Similarly, for triples of type $(3,3,2)$, 
we have $\ell_\sigma=3$, the numerators are multiples 
of $3$ and $u$ is integral.
By Theorem~\ref{thm:main} this contradicts the
terminality of~$X$ and we are left with the configurations
\begin{center}
\begin{tabular}{r|r|r|r|r}
$l_{01}$ & $l_{02}$ & $l_{11}$ & $l_{12}$ & $l_{21}$ 
\\
\hline
$2$ & $1$ & $2$ & $1$ & $l_{21}$ 
\\
$2$ & $1$ & $2$ & $2$ & $l_{21}$ 
\\
$2$ & $2$ & $2$ & $2$ & $l_{21}$ 
\\
$l_{01}$ & $l_{02}$ & $2$ & $1$ & $2$ 
\\
$l_{01}$ & $l_{02}$ & $2$ & $2$ & $2$ 
\end{tabular}
\end{center}
In each of the cases, we detect a lattice
point on an edge of $A_X^c$ located in a 
leaf, contradicting again terminality.
The procedure is the same
for all configurations; 
we treat exemplarily the first one. 
There, after suitable admissible operations, 
the matrix $P$ is of the form
$$ 
P 
\ = \ 
\left[
\begin{array}{rrrrr}
-2 & -1 & 2 & 1 & 0
\\
 -2 & -1 & 0 & 0 & l_{21} 
\\
1 & 0 & d_{111} & 0 & d_{121}
\\
0 & 0 & d_{211} & 0 & d_{221} 
\end{array}
\right].
$$
According to Corollary~\ref{thm:cpl1antican},
the vertices of the support of the lineality 
part $A_{X,0}^c$ of the anticanonical complex 
$A_{X}^c$ are given by  
\begin{eqnarray*}
u_1 
& := & 
\left(           
0,\ 0, \
\frac{l_{21}}{2}+\frac{l_{21}}{2}d_{111}+d_{121}
,\
\frac{l_{21}}{2}d_{211}+d_{221}
\right),
\\
u_2
& := &
\left( 
0,\ 0, \
\frac{l_{21}d_{111}+2d_{121}}{l_{21} +2}
,\
\frac{l_{21}d_{211}+2d_{221}}{l_{21} + 2}
\right), 
\\
u_3 
& := &
\left(          
0,\ 0, \
\frac{l_{21}+2d_{121}}{l_{21} + 2}
,\
\frac{2d_{221}}{l_{21} + 2}
\right),
\\
u_4 
& := &
\left(          
0,\ 0, \
\frac{d_{121}}{l_{21} + 1}
,\
\frac{d_{221}}{l_{21} + 1}
\right).
\end{eqnarray*}
Note that $l_{21}$ is odd, because otherwise $u_1$ would 
be a lattice point.
Using once again Corollary~\ref{thm:cpl1antican}, we obtain the 
following explicit description of the second leaf:
$$
\vert A_{X}^c \sqcap \tau_2 \vert  
\ = \ 
\conv(v_{21},u_1,u_2,u_3,u_4),
$$
where $v_{21}$ denotes the last column of the matrix $P$.
Using the fact that $l_{21}$ is odd, we see that on 
the edge connecting $v_{21}$ to $u_1$
lies at least one lattice point, namely
$$
\frac{l_{21}-1}{l_{21}} u_1 + \frac{1}{l_{21}} v_{21} 
\ = \ 
\left( 
0,\ 1, \
d_{121}+(d_{111}+1)\frac{l_{21}-1}{2}
,\
d_{221}+d_{211}\frac{l_{21}-1}{2}
\right).
$$
Similarly, we find in the remaining cases such a 
point on an edge of $A_{X}^c$ connecting a half-integral vertex 
of $A_{X,0}^c$ with a column of $P$ containing one 
of the not yet fixed~$l_{ij}$.
\end{proof}

As a consequence of Proposition~\ref{prop:higherplt}, 
we can focus our search for terminal varieties 
$X(A,P)$ on defining matrices $P$ of the following type.

\begin{setting}
\label{set:tower221bounds}
Let $X = X(A,P)$ be a non-toric terminal
$\QQ$-factorial Fano threefold of Picard 
number one,
such that $P$ is irredundant
with $r=2$, $m=0$ and $n=5$, where $n_0 = n_1 = 2$, 
$n_2 = 1$.
Assume that $l_{01} = l_{02} = 1$ holds.
Then, after suitable admissible operations, $P$ is of 
the form
$$ 
P 
\ = \ 
\left[
\begin{array}{rrrrr}
-1 & -1 & l_{11} & l_{12} & 0
\\
 -1 & -1 & 0 & 0 & l_{21} 
\\
0 & 1 & d_{111} & d_{112} & d_{121}
\\
0 & 0 & d_{211} & d_{212} & d_{221} 
\end{array}
\right],
$$
where $l_{11} \ge l_{12}$ and $l_{21} \ge 2$ hold.
Moreover, denoting by $P_{ij}$ the matrix obtained by
removing the column $v_{ij}$ from $P$, we have 
positive \emph{weights}
$$
w_{01} \ := \ \det(P_{01}),
\qquad
w_{02} \ := \ -\det(P_{02}),
$$
$$
w_{11} \ := \ \det(P_{11}),
\quad
w_{12} \ := \ -\det(P_{12}),
\quad
w_{21} \ := \ \det(P_{21}).
$$
Observe that the weight vector 
$(w_{01},w_{02},w_{11},w_{12},w_{21})$ 
lies in the kernel of $P$.
The last three weights are explicitly given by 
$$ 
w_{11} \ = \ -l_{21}d_{212} - l_{12}d_{221},
\quad
w_{12} \ = \ l_{21}d_{211} + l_{11}d_{221},
\quad
w_{21} \ = \ -l_{11}d_{212} + l_{12}d_{211}
$$
and the first two weights can be expressed in a compact 
form in terms of the others as follows:
$$
w_{02} \ = \ -d_{111}w_{11}-d_{112}w_{12}-d_{121}w_{21},
\qquad 
w_{01} \ = \ l_{21}w_{21} - w_{02} .
$$
\end{setting}

\begin{remark}
\label{rem:tower221bounds0}
In Setting~\ref{set:tower221bounds},
we can achieve by further admissible 
operations without changing the shape 
of $P$ the following for the entries 
of the third and fourth row of $P$: 
$$ 
0 \ \le \ d_{121}, d_{221} \ < \ l_{21},
\qquad
d_{121}  \ < \ d_{221} \text{ if } d_{221} \ne 0,
\qquad
0 \ \leq \ d_{112} \ < \ w_{11},
$$
$$
-\frac{(l_{21}+d_{121})w_{21}+d_{112}w_{12}}{w_{11}}
\ < \ 
d_{111}  
\ < \ 
-\frac{d_{121}w_{21}+d_{112}w_{12}}{w_{11}}.
$$
For the third estimate we add
a suitable multiple of $d_{221}(p_1 - p_2) + l_{21}p_4$  
to~$p_3$, where $p_i$ denotes the $i$-th row of $P$
(this preserves the first two estimates).
The inequalities for $d_{111}$ follow directly from 
$w_{02} >0$ and $w_{01}  > 0$.
\end{remark}

A first series of bounds on the entries of the defining matrix
$P$ is derived from the fact that, by terminality, 
the lineality part $A_{X,0}^c$ of the anticanonical complex 
$A_X^c$ has the origin as its only lattice point;
we also write $A_{X,0}^c$ for the support of the 
lineality part, which in our situation is a rational 
two-dimensional polytope. Here is how it precisely looks.

\begin{lemma}
\label{lem:tower221bounds1}
Let $X = X(A,P)$ be as in Setting~\ref{set:tower221bounds}.
The vertices of $A_{X,0}^c$ regarded as a subset of the lineality space 
$\QQ^2$ of the tropical variety are  
\begin{eqnarray*}
u_1 
& := & 
\left[           
\frac{l_{21} d_{111} + l_{11} d_{121}}{l_{21} + l_{11}}
,\
\frac{l_{21}d_{211} + l_{11}d_{221}}{l_{21} + l_{11}}
\right],
\\
u_2
& := &
\left[           
\frac{l_{11}l_{21} + l_{21} d_{111} + l_{11} d_{121}}{l_{21} + l_{11}}
,\
\frac{l_{21}d_{211} + l_{11}d_{221}}{l_{21} + l_{11}}
\right], 
\\
u_3 
& := &
\left[           
\frac{l_{21} d_{112} + l_{12} d_{121}}{l_{21} + l_{12}}
,\
\frac{l_{21}d_{212} + l_{12}d_{221}}{l_{21} + l_{12}}
\right],
\\
u_4 
& := &
\left[           
\frac{l_{12}l_{21} + l_{21} d_{112} + l_{12} d_{121}}{l_{21} + l_{12}}
,\
\frac{l_{21}d_{212} + l_{12}d_{221}}{l_{21} + l_{12}}
\right].
\end{eqnarray*}
\end{lemma}

\begin{proof}
We just compute the lineality part $A_{X,0}^c$ of the 
anticanonical complex $A_{X}^c$ according to Corollary~\ref{thm:cpl1antican}.
\end{proof}

\begin{remark}
\label{rem:trapez}
Observe that $A_{X,0}^c$ as described in 
Lemma~\ref{lem:tower221bounds1}
is a trapezoid.
The edges $g_1 := \overline{u_1u_2}$ and 
$g_2 := \overline{u_3u_4}$ are parallel to the $x$-axis 
and the remaining two edges are $\overline{u_1u_3}$ 
and $\overline{u_2u_4}$.
Length and $y$-value $h(g_i)$ of the line segments $g_i$ 
are
$$ 
\vert g_1 \vert 
\ = \ 
\frac{l_{11}l_{21}}{l_ {11}+l_{21}},
\quad
h(g_1) 
\ = \ 
\frac{w_{12}}{l_{11} + l_{21}},
\qquad
\vert g_2 \vert 
\ = \ 
\frac{l_{12}l_{21}}{l_ {12}+l_{21}},
\quad
h(g_2) 
\ = \ 
-\frac{w_{11}}{l_{12} + l_{21}}.
$$
Since we assume $l_{11}\ge l_{12}$ in Setting~\ref{set:tower221bounds}, 
the lower segment $g_2$ is shorter than the upper segment $g_1$.
Note that the values $\vert g_i \vert$ and $h(g_i)$ are invariant
under admissible row operations of type~\ref{remark:admissibleops}~(iii).
\end{remark}

\begin{lemma}
\label{lem:tower221bounds2}
Let $X = X(A,P)$ be as in Setting~\ref{set:tower221bounds}.
Let $h := h(g_1)-h(g_2)$ denote the total height of the 
trapezoid $A_{X,0}^c$. Then we have 
$$
\frac{l_{12}l_{21}}{l_{12}+l_{21}} 
\ < \ 
2,
\qquad
\frac{l_{11}l_{21}}{l_{11}+l_{21}}
\ < \ 
\frac{2(l_{12}+l_{21})-l_{12}l_{21}}{w_{11}}
\cdot 
h 
+
\frac{l_{12}l_{21}}{l_{12}+l_{21}}.
$$
Moreover, one has the following estimates
$$
w_{01} \ <  \ w_{11} + w_{12} + w_{21},
\qquad\qquad
w_{02} \ <  \ w_{11} + w_{12} + w_{21}.
$$
\end{lemma}

\begin{proof}
For the first inequality, note that the 
lower bounding segment $g_2$ of $A_{X,0}^c$ is of 
length at most 2, because otherwise the 
segment $A_{X,0}^c \cap \{y=0\}$ is of length at least 
2 as well, which would imply existence of lattice points different 
from the origin in $A_X^c$. 
Similarly, since $A_{X,0}^c \cap \{y=0\}$ has length strictly smaller 
than 2, we arrive at the second inequality: 
$$
\vert g_1 \vert
\ < \ 
\frac{2- \vert g_2 \vert}{\vert h(g_2) \vert} \cdot  h 
+ \vert g_2 \vert.
$$
Explicitly computing $A_{X,0}^c \cap \{y=0\}$
gives the bounding $x$-values
$-w_{01} / (w_{11} + w_{12} + w_{21})$ and 
$w_{02} / (w_{11} + w_{12} + w_{21})$. 
Since the origin is the only lattice point in $A_{X,0}^c \cap \{y=0\}$,
we arrive at estimates number three and four.
\end{proof}

\begin{lemma}
\label{lem:tower221bounds7}
Let $X = X(A,P)$ be as in Setting~\ref{set:tower221bounds}.
If $l_{21} \ge 3$ holds, then we obtain the estimate
$$ 
l_{12} \ < \ \frac{l_{21}+2}{l_{21}-2} \ \le \ 5.
$$
\end{lemma}

\begin{proof}
Estimates three and four from Lemma~\ref{lem:tower221bounds2} imply
$$ 
l_{21}w_{21} 
\ = \ 
l_{11}w_{11} + l_{12}w_{12} 
\ = \ 
w_{01} + w_{02} 
\ < \
2w_{11} + 2w_{12} + 2w_{21}.
$$
We deduce
$$
(l_{11}-2)w_{11} + (l_{12}-2)w_{12}  
\ < \ 
2w_{21},
\qquad
(l_{21}-2)w_{21}
\ < \ 
2w_{11} + 2w_{12}.
$$
Using $l_{21} \ge 3$ we obtain
$$ 
(l_{11}-2)w_{11} + (l_{12}-2)w_{12}  
\ < \ 
\frac{4}{l_{21}-2}w_{11} + \frac{4}{l_{21}-2}w_{12},
$$
which implies
$$
l_{11}w_{11} + l_{12}w_{12}  
\ < \ 
\frac{l_{21}+2}{l_{21}-2}w_{11} + \frac{l_{21}+2}{l_{21}-2}w_{12}
$$
and in particular
$$
l_{12}
\ < \ 
\frac{l_{21}+2}{l_{21}-2}.
$$
\end{proof}

\begin{remark}
\label{rem:tower221bounds2}
Let $X = X(A,P)$ be as in Setting~\ref{set:tower221bounds}.
For $c>0$ the assumption $h(g_2) > -c$ leads to 
$$ 
-c - \frac{l_{12}}{l_{21}}(c + d_{221})
\ < \
d_{212}
\ < \ 
0.
$$
\end{remark}

\begin{remark}
\label{rem:tower221bounds3}
Let $X = X(A,P)$ be as in Setting~\ref{set:tower221bounds}.
If $h(g_1) < 1$ holds, then we have 
$$
- \frac{l_{11}}{l_{21}}d_{221}
\ < \ 
d_{211}
\ < \  
- \frac{l_{11}}{l_{21}}d_{221} + 1 + \frac{l_{11}}{l_{21}}.
$$
\end{remark}

\begin{lemma}
\label{lem:tower221bounds3}
Let $X = X(A,P)$ be as in Setting~\ref{set:tower221bounds}.
Assume $l_{21}\ge3$.
If $h(g_1) < 1$ and $h(g_2) > -2$ hold, then we have 
$$ 
l_{11} \ < \ 2 \frac{l_{21}}{l_{21}-2}.
$$
This bounds $l_{11}$ in terms of $l_{21}$ in the case 
$h(g_1) < 1$ and $h(g_2) > -2$. 
In particular, we then have $l_{11} \le 5$ and we have $l_{11} \le 2$ 
as soon as $l_{21} \ge 6$.
\end{lemma}

\begin{proof}
Observe that $w_{01}+w_{02} = l_{11}w_{11} + l_{12}w_{12}$.
Thus, the third and the fourth inequalities of Lemma~\ref{lem:tower221bounds2}
give us the condition
$$ 
l_{11}w_{11} + l_{12}w_{12} 
\ < \ 
2w_{11} + 2w_{12} + 2w_{21}.
$$
We arrive at the assertion by
writing this out and estimating $d_{212}$ as 
well as $d_{211}$ according 
to Remarks~\ref{rem:tower221bounds2} and~\ref{rem:tower221bounds3}.
\end{proof}

\begin{lemma}
\label{lem:tower221bounds8}
Let $X = X(A,P)$ be as in Setting~\ref{set:tower221bounds}.
Suppose that $h(g_1) < 1$ and $h(g_2)  \le -c$ holds for some 
$c \in \ZZ_{\ge 2}$.
Then we have $l_{12} = 1$ and moreover 
$$ 
\frac{l_{11}l_{21}}{l_{11}+l_{21}} 
\ < \ 
\frac{c+1}{c-1} - \frac{2}{c-1} \cdot \frac{l_{21}}{1+l_{21}}.
$$
\end{lemma}
\begin{proof}
Since $h(g_2) \le -1$ holds, we must 
have $\vert g_2 \vert < 1$ and thus 
obtain $l_{12} = 1$. 
The line segment 
$A_{X,0}^c \cap \{y=-1\}$ is of length 
strictly smaller than $1$
and $A_{X,0}^c \cap \{y=-c\}$ is of length
at least $\vert g_2 \vert$.
Since  $h(g_1) < 1$ holds, we conclude 
$$ 
\frac{l_{11}l_{21}}{l_{11}+l_{21}}
\ = \ 
\vert g_1 \vert 
\ < \ 
\frac{1 - \vert g_2 \vert}{c-1}
(1 + h(g_1)) + 1
\ < \ 
\frac{c+1}{c-1} - \frac{2}{c-1} \cdot \frac{l_{21}}{1+l_{21}}.
$$
\end{proof}

\begin{remark}
\label{rem:tower221bounds4}
Let $X = X(A,P)$ be as in Setting~\ref{set:tower221bounds}.
Assume $l_{12} = 1$ and $d_{112} = d_{212} = 0$. 
Then $w_{11} > 0$ and $w_{12} > 0$ imply 
$$
0 \ < \ d_{211},
\qquad\qquad
-\frac{l_{21}}{l_{11}}d_{211} 
\ < \ 
d_{221} 
\ < \ 0.
$$
Moreover, the conditions $h(g_1) < 1$ and $h(g_2) > -c$ are equivalent 
to the following conditions
$$ 
d_{211} \ < \ - \frac{l_{11}}{l_{21}}(d_{221} - 1) + 1,
\qquad\qquad
d_{221} \ > \ -c(l_{21}+1).
$$
\end{remark}

\begin{lemma}
\label{lem:tower221bounds4}
Let $X = X(A,P)$ be as in Setting~\ref{set:tower221bounds}.
Suppose that $h(g_1) \ge 1$ holds. Then  either $l_{11} = l_{12} = 1$ 
or $l_{11} = l_{21} = 2$ hold.
\end{lemma}

\begin{proof}
First observe that in this case, 
the segment $A_{X,0}^c \cap \{y=1\}$ can be of
length at most 1, because otherwise we have 
lattice points different from the origin 
and the vertices in $A_{X}^c$. This means 
$l_{11} = 1$ or $l_{11} = l_{21} = 2$.
\end{proof}

A second series of estimates makes use of the whole 
anticanonical complex $A_X^c$. 
The strategy is to detect via $A_X^c$ suitable 
three-dimensional lattice simplices with precisely 
one interior lattice point and to use the volume 
bounds given in~\cite{AvKrNi} in order to control the 
entries of the defining matrix~$P$.
We will distinguish several cases, using the notation
of Remark~\ref{rem:trapez}.

\begin{proposition}
\label{prop:l11-l12-1}
Let $X = X(A,P)$ be as in Setting~\ref{set:tower221bounds}.
Suppose $l_{11} = l_{12} = 1$.
Then we achieve by admissible operations
$d_{112} = d_{212} = 0$ 
and obtain the estimates
$$ 
3 \ \le  \ (l_{21}+1)d_{211} \ \le \ 72,
\qquad\qquad
0 \ \le \ d_{111} \ < \ d_{211},
$$
$$
-d_{211}l_{21} \ < \ d_{221} \ < \ 0,
\qquad\qquad
\frac{d_{111}d_{221}}{d_{211}} - l_{21} \ < \ d_{121} \ < \ 0.
$$
\end{proposition}

\begin{proof}
Consider the convex hull $C'$ of $A_{X,0}^c$ and 
$v_{21}$. We may regard $C'$ as a polytope in $\QQ^3$ 
by omitting the first coordinate. Then $C'$ is contained in the 
polytope $C$ with the vertices
$$ 
(l_{21},d_{121},d_{221}),
\quad
(-1,d_{111},d_{211}),
\quad
(-1,1+d_{111},d_{211}),
\quad
(-1,0,0),
\quad
(-1,1,0).
$$
Now, $C$ is a lattice polytope having $(0,0,0)$ as the 
only interior lattice point. 
There are precisely two ways to write $C$ as a union of 
two simplices,
$$
C \ = \ C_1\cup C_2 \ = \ C_3 \cup C_4.
$$
For each of these simplices, the volume is 
$\mathrm{vol}(C_j) = (l_{21}+1)d_{211}/6$.
If the origin lies in the interior of one of the $C_j$,
then, according to~\cite[Thm.~2.2]{AvKrNi}, its volume 
is at most~$12$. 
This gives the bound
$$
(l_{21}+1)d_{211}
\ = \ 
6  \cdot \mathrm{vol} (C_j)
\ \leq \ 
72.
$$ 
The remaining estimates follow
from positivity of the weights $w_{ij}$.
If the origin lies in $C_1 \cap C_2 \cap C_3 \cap C_4$,
then we must have
$$
l_{21} \ = \ 2, 
\qquad
d_{211} \ = \ -d_{221},
\qquad
d_{111} \ = \ -1-d_{121}.
$$
Positivity of the weights provides the inequalities 
$d_{121},d_{221}<0$.
Since the origin is the only lattice point in $A_{X,0}^c$,
we get $d_{121}>-5$ and $d_{221}>3(d_{121}+1)$,
which altogether fulfill the estimates of this proposition.
\end{proof}

\begin{proposition}
\label{prop:l21-2}
Let $X = X(A,P)$ be as in Setting~\ref{set:tower221bounds}.
Suppose $l_{21} = 2$. 
\begin{enumerate}
\item
If $h(g_1)<1$ and $h(g_2)>-1$ hold, turn
$P$ by means of admissible operations into 
the shape of Remark~\ref{rem:tower221bounds0}.
Then we are in one of the  situations:
\begin{enumerate}
\item
$d_{121}=1$, $d_{221}=0$, $(2+l_{12})d_{211}+(2+l_{11})(-d_{212}) \le 36$,
\item
$d_{121}=0$, $d_{221}=1$, $(l_{11}-l_{12})+(2+l_{12})d_{211}+(2+l_{11})(-d_{212}) \le 36$.
\end{enumerate}
In both situations the remaining entries $d_{111}$, $d_{112}$ 
are bounded according to Remark~\ref{rem:tower221bounds0}.
\item
If $h(g_1)<1$ and $h(g_2)\le-1$ hold, then we have $l_{12}=1$.
Moreover adjusting $d_{112}=d_{212}=0$ by admissible operations,
we arrive in one of the following three situations:
\begin{enumerate}
\item
$l_{11}=1$ holds and Proposition~\ref{prop:l11-l12-1} applies.
\item
$l_{11}=2$ holds and we have estimates
$$
-6 \leq d_{221} \leq -3, \qquad\qquad d_{211}=1-d_{221}.
$$
\item
$3\le l_{11}< 140$ holds and we have estimates
$$
\qquad\qquad\qquad 
\frac{-5l_{11}+2}{l_{11}-2} < d_{221} \le -3, 
\qquad
-\frac{l_{11}}{2}d_{221} < d_{211} < -\frac{l_{11}}{2}d_{221}+\frac{l_{11}}{2}.
$$
\end{enumerate}
In both cases (b) and (c), the remaining entries of the defining 
matrix $P$ are bounded by
$$
\qquad\quad 
0 \le d_{121} < -d_{221}, 
\qquad\quad 
\frac{d_{121}d_{211}}{d_{221}}+2\frac{d_{211}}{d_{221}} 
< d_{111} < 
\frac{d_{121}d_{211}}{d_{221}}.
$$
\item
If $h(g_1)\ge1$ holds, then we have $l_{11}=l_{12}=1$ and 
Proposition~\ref{prop:l11-l12-1} applies.
\end{enumerate}
\end{proposition}

\begin{proof}
We prove~(i).
First observe that Remark~\ref{rem:tower221bounds0} 
yields $d_{221}\in\{0,1\}$ because of $l_{21} = 2$.
If $d_{221}=1$ holds, then Remark~\ref{rem:tower221bounds0} 
implies $d_{121}=0$.
If $d_{221}=0$ holds, then we must have $d_{121}=1$ 
because $v_{21}$ is a primitive lattice point.
This leads to cases (a) and (b) as the only 
possibilities.
For the estimate of case (a), we look at the 
lattice simplex 
$C_1$ given in $\QQ^3$ as the convex hull of 
the following points:
$$
(l_{11},d_{111},d_{211}), 
\quad (l_{12},d_{112},d_{212}), 
\quad (-2,1,0), 
\quad (-2,3,0).
$$
To obtain the estimate of case (b), we look at the 
lattice simplex $C_2$  in $\QQ^3$ given as the convex 
hull of the following points:
$$
(l_{11},d_{111},d_{211}), 
\quad (l_{12},d_{112},d_{212}), 
\quad (-2,0,1), 
\quad (-2,2,1).
$$
For the volumes, we obtain in both cases  
$\text{vol}(C_i) = (w_{11}+w_{12}+w_{21})/3$.
Now, put the leaf $A_X^c \cap \tau_1$ of the 
anticanonical complex into $\QQ^3$ by removing the 
second coordinate (which always equals zero) 
from its points.
For $a = 0,-1,-2$, consider 
$$
H_{a}^+ \ := \{(x,y,z); \; x \ge a\} \ \subseteq \ \QQ^3,
\qquad
H_{a}^0 \ := \{(x,y,z); \; x = a\} \ \subseteq \ \QQ^3.
$$
Then $C_i \cap H_{0}^+$ equals $A_X^c \cap \tau_1$
and $H_0^0$ cuts out the lineality part $A_{X,0}^c$.
In particular, by terminality of $X$ and Theorem~\ref{thm:main},
the intersection $C_i \cap H_{0}^+$ has no interior lattice point
and inside $C_i \cap H_{0}^0$ the origin is the only 
lattice point.
The intersection $C_1 \cap H_{-1}^0$ has the vertices
$$
\Bigl(-1,\frac{l_{11}+d_{111}+1}{l_{11}+2},\frac{d_{211}}{l_{11}+2}\Bigr),
\qquad
\Bigl(-1,\frac{3l_{11}+d_{111}+3}{l_{11}+2},\frac{d_{211}}{l_{11}+2}\Bigr),
$$
$$
\Bigl(-1,\frac{l_{12}+d_{112}+1}{l_{12}+2},\frac{d_{212}}{l_{12}+2}\Bigr),
\qquad
\Bigl(-1,\frac{3l_{12}+d_{112}+3}{l_{12}+2},\frac{d_{212}}{l_{12}+2}\Bigr),
$$
while the intersection $C_2 \cap H_{-1}^0$ has the vertices
$$
\Bigl(-1,\frac{d_{111}}{l_{11}+2},\frac{l_{11}+d_{211}+1}{l_{11}+2}\Bigr),
\qquad
\Bigl(-1,\frac{2l_{11}+d_{111}+2}{l_{11}+2},\frac{l_{11}+d_{211}+1}{l_{11}+2}\Bigr),
$$
$$
\Bigl(-1,\frac{d_{112}}{l_{12}+2},\frac{l_{12}+d_{212}+1}{l_{12}+2}\Bigr),
\qquad
\Bigl(-1,\frac{2l_{12}+d_{112}+2}{l_{12}+2},\frac{l_{12}+d_{212}+1}{l_{12}+2}\Bigr),
$$
The inequalities $h(g_1)<1$ and $h(g_2)>-1$
together with the positivity of the weights ensure that 
the points of $C_i \cap H_{-1}^0$ never have an integral 
$z$-value.
We can conclude that $C_i \subseteq H_{-2}^+$ 
has the origin as its only interior lattice point.
Applying the bound $\text{vol}(C_i)\leq12$ from~\cite[Thm.~2.2]{AvKrNi} 
and writing down the involved weights explicitly
we arrive at the assertion.

We turn to~(ii). By Lemma~\ref{lem:tower221bounds8} we have $l_{12}=1$.
By admissible operations, we achieve $d_{112}=d_{212}=0$.
If $l_{11}=1$ holds,
we can apply Proposition~\ref{prop:l11-l12-1}. 
Let $l_{11}\ge2$.
For $l_{11}\ge 3$, the positivity of the weights and
the constraints on the heights together with 
suitable admissible operations lead to 
all the bounds for the $d_{ijk}$ stated in~(c)
except for the lower bound on $d_{221}$.
For that, observe that the segment $A_{X,0}^c\cap\{y=-1\}$
has to be of length strictly smaller than $1$ 
and conclude
$$
\frac{-5l_{11}+2}{l_{11}-2} \ < \ d_{221}.
$$
The next step is to bound $l_{11}$.
For this, we consider the simplex $D \subseteq \QQ^3$ 
given as the convex hull of following points
$$
(l_{11},d_{111},d_{211}), \quad (1,0,0), \quad (-2,d_{121},d_{221}), \quad (-2,d_{121}+2,d_{221}).
$$
Now, put the leaf $A_X^c \cap \tau_1$ of the anticanonical complex into $\QQ^3$
by removing the second coordinate (which always equals zero).
With the same notation as in part (i) of the proof,
we see that $D \cap H_{0}^+$ equals $A_X^c \cap \tau_1$
and $H_0^0$ cuts out the lineality part $A_{X,0}^c$.
For $l_{11}\ge 10$ the only possible values for $d_{221}$ are $-3,-4,-5$.
Moreover we already have $0\leq d_{121} < -d_{221}$.
Thus the allowed pairs $(d_{121},d_{221})$ are
$$
(0,-3),\quad (1,-3),\quad (2,-3),\quad (1,-4),\quad (3,-4),
$$
$$
(0,-5),\quad (1,-5),\quad (2,-5),\quad (3,-5),\quad (4,-5).
$$
Actually all of them, except the fourth, the seventh and the eigth,
already provide an inner lattice point in the lineality part.
Going through the remaining three pairs we are able to determine
the inner lattice points of $D$ other than the origin.
These points can now only lie in $H_{-1}^0$.
By finding a simplex with exactly one inner lattice point 
and using~\cite[Thm.~2.2]{AvKrNi} we obtain $l_{21} \le 140$.
Here we treat the pair $(1,-4)$ as an example, since it provides the worst estimate.
In this case $D$ has vertices
$$
v_{11}=(l_{11},d_{111},d_{211}), \quad v_{12}=(1,0,0), \quad a_1:=(-2,1,-4), \quad a_2:=(-2,3,-4),
$$
and we get $p:=(-1,1,-2)$ as only inner lattice point other than the origin.
We define simplices $D_1 := \conv(p,v_{11},v_{12},a_1)$ and $D_2 := \conv(p,v_{11},v_{12},a_2)$.
The origin lies in one of the two simplices $D_i$.
Bounding their volumes by $12$ according to~\cite[Thm.~2.2]{AvKrNi}
we obtain $l_{11}\leq 70$ if $0\in D_1^\circ$ and $l_{11}\leq 140$ if $0\in D_2^\circ$.
Note that the origin cannot lie in $D_1 \cap D_2 = \conv(p,v_{11},v_{12})$:
we would have $d_{211}=-2d_{111}$, but then
terminality would require $\gcd(d_{111},d_{211})=1$,
which in turn fixes $d_{111}=-1$ and $d_{211}=2$.
Now with the second estimate of (c) $l_{11}<1$ must hold,
a clear contradiction to $l_{11}\ge2$.

Now we turn to the case $l_{11}=2$ and prove the estimates of (b).
Here $u_1$ and $u_2$ are half-integral points,
therefore we have $h(g_1)=1/2$, which implies $d_{211}+d_{221}=1$.
The constraint $h(g_2)\le -1$ is equivalent to $d_{221}\le -3$.
Estimates on $d_{111}$ and $d_{121}$ are found 
by positivity of the weights and admissible operations.
For the lower bound on $d_{221}$ we note that 
$u_3$ lies under the bisection of the fourth quadrant.
Requiring that no lattice point lies in $A_{X,0}^c$ except for the origin
only leaves a confined area to place $g_2$, namely 
$h(g_2)\ge -2$ must hold.
This provides the bound $d_{221}\ge -6$.

Let us verify~(iii). By Lemma~\ref{lem:tower221bounds4}
we have $(l_{11},l_{12}) \in \{ (1,1) , (2,1) , (2,2) \}$.
If both exponents are equal $1$, then Proposition~\ref{prop:l11-l12-1} applies straightforward.
If both exponents are equal $2$, then $\vert g_1 \vert = \vert g_2 \vert = 1$.
This implies that the segment $A_{X,0}^c \cap \{y=1\}$
is of length one and hence contains at least one lattice point.
Lastly we show that the case $(l_{11},l_{12})=(2,1)$ is also not possible.
Here it holds $\vert g_1 \vert = 1$ and
two of the vertices are
$$
u_1 = \Bigl( \frac{1}{2}d_{111}+\frac{1}{2}d_{121} \,,\, \frac{1}{2}d_{211}+\frac{1}{2}d_{221} \Bigr),
\qquad\quad u_2 = u_1+(1,0).
$$
We assume $h(g_1)$ to be non-integral,
otherwise we would have a lattice point on $g_1$ itself.
Nonetheless an integral point $p$ is always in the lineality part, 
precisely at the height $h(g_1)-1/2$
and it can be given explicitly as $p:=\alpha u_1 +\beta u_2$ where
$$
\alpha := -k-\frac{d_{111}+d_{121}+2}{2(d_{211}+d_{221})}, \qquad\quad
\beta := 1+k+\frac{d_{111}+d_{121}}{2(d_{211}+d_{221})}
$$
for an appropriate $k\in\ZZ_{\ge0}$ that makes $0\leq \alpha,\beta <1$.
Then we have
$$
p = \Bigl( \frac{1}{2}d_{111}+\frac{1}{2}d_{121}+k+1 \,,\, h(g_1)-\frac{1}{2} \Bigr),
$$
which is an integral point since we can always assume
$d_{111}$ and $d_{121}$ to have the same parity.
\end{proof}

\begin{proposition}
\label{prop:l21-ge3}
Let $X = X(A,P)$ be as in Setting~\ref{set:tower221bounds}.
Suppose $l_{21} \ge 3$.
\begin{enumerate}
\item
If $h(g_1)<1$ and $h(g_2)>-2$ hold, then we are in one of 
the following three situations:
\begin{enumerate}
\item
We have $3 \le l_{21} \le 5$ and the other $l_{ij}$ are bounded 
according to the table
\begin{center}
\begin{tabular}{r|r|r|r}
$l_{21}$ & $3$ & $4$ & $5$
\\
\hline
$l_{12}$ & $\le 4$ & $\le 2$ & $\le 2$
\\
\hline
$l_{11}$ & $\le 5$ & $\le 3$ & $\le 2$
\end{tabular}
\end{center}
In this case turn $P$ by means of admissible operations
into the shape of Remark~\ref{rem:tower221bounds0}.
Then we have $0 \le d_{121},d_{221} < l_{21}$ and the estimates
$$
\qquad\qquad\qquad
-2-\frac{l_{12}}{l_{21}}(d_{221}+2)< d_{212} < 0,
\quad
-\frac{l_{11}}{l_{21}}d_{221} < d_{211} < -\frac{l_{11}}{l_{21}}d_{221} +1 + \frac{l_{11}}{l_{21}}
$$
and the remaining two entries $d_{111}$, $d_{112}$ 
are bounded according to Remark~\ref{rem:tower221bounds0}.
\item
We have $6 \leq l_{21}$ and $l_{11}=l_{12}=1$.
Then all entries $d_{ijk}$ can be bounded 
according to Proposition~\ref{prop:l11-l12-1}.
\item
We have $6 \leq l_{21}$, $l_{11}=2$ and $l_{12}=1$.
Then we achieve $d_{112}=d_{212}=0$ by suitable admissible
operations and values and bounds for the remaining entries 
are given by the table
\begin{center}
\begin{tabular}{r|r|r|r|r|r|r|r}
$d_{111}$ & $0$ & $1$ & $1$ & $1$ & $1$ & $2$ & $2$
\\
$d_{211}$ & $1$ & $2$ & $3$ & $4$ & $5$ & $3$ & $5$
\\
\hline
$l_{21}$ & $\le 141$ & $\le 71$ & $\le 71$ & $\le 179$ & $\le 177$ & $\le 137$ & $\le 143$
\end{tabular}
\end{center}
and by the estimates
$$
\qquad\qquad  -2(l_{21}+1) < d_{221} < 0,  \qquad 
\frac{d_{111}d_{221}}{d_{211}} -l_{21} < d_{121} < \frac{d_{111}d_{221}}{d_{211}}.
$$
\end{enumerate}
\item
If $h(g_1)<1$ and $h(g_2)\le-2$ hold, then we are in one of 
the following two situations:
\begin{enumerate}
\item
We have $l_{11}=l_{12}=1$. Then $l_{21}$ and the entries $d_{ijk}$ 
can be bounded according to Proposition~\ref{prop:l11-l12-1}.
\item
We have $l_{11}=2$, $l_{12}=1$ and $l_{21}=3,4$.
Then we achieve $d_{112}=d_{212}=0$ by admissible 
operations and obtain the following estimates
$$
\qquad\qquad  -4(l_{21}+1) < d_{221} < 0,
\qquad -\frac{2}{l_{21}}d_{221} < d_{211} < -\frac{2}{l_{21}}(d_{221}-1)+1,
$$
$$
\qquad\qquad 0 \le d_{121} < -d_{221}, \qquad
\frac{d_{211}(d_{121}+l_{21})}{d_{221}} < d_{111} < \frac{d_{211}d_{121}}{d_{221}}.
$$
\end{enumerate}
\item
If $h(g_1) \ge 1$ holds, then we have $l_{11}=l_{12}=1$ and Proposition~\ref{prop:l11-l12-1} applies.
\end{enumerate}
\end{proposition}

\begin{proof}
Let us verify~(i). 
Lemmas~\ref{lem:tower221bounds7} and~\ref{lem:tower221bounds3}
provide us bounds on $l_{11}$ and $l_{12}$ in terms of $l_{21}$, namely
those from the table of case~(a) if $l_{21}<6$, 
otherwise $l_{11}=1,2$ and $l_{12}=1$.
The other estimates of case~(a) follow directly from
Remarks~\ref{rem:tower221bounds0}, \ref{rem:tower221bounds2} 
and~\ref{rem:tower221bounds3}.
From now on we have $l_{21}\ge 6$ and $l_{12}=1$,
thus we assume $d_{112}=d_{212}=0$ by admissible operations.
If $l_{11}=1$ then we are in case~(b) and 
Proposition~\ref{prop:l11-l12-1} applies.
If $l_{11}=2$ then we have to prove the estimates of case~(c).
Writing down explicitly the inequalities $h(g_1)<1$ and $h(g_2)>-2$ as well as
the positivity of the weights already gives us
the bounds for $d_{121}$ and $d_{221}$ and the following estimates
$$
-\frac{2d_{221}}{l_{21}} < d_{211} < -\frac{2d_{221}}{l_{21}}+\frac{l_{21}+2}{l_{21}},
\qquad 0 \le d_{111} < d_{211}.
$$
All we are left to find is an upper bound for $l_{21}$.
Note that by substituting the lower estimate for $d_{221}$
in the upper estimate for $d_{211}$ one obtains
$$
0 < d_{211} < 5 + \frac{6}{l_{21}} \leq 6.
$$
Thus we have a finite range (independent from $l_{21}$)
for $d_{211}$ and therefore for $d_{111}$ too, namely
$$
d_{111} \in \{0,1,2\}, \qquad\quad d_{111} < d_{211} \le 5.
$$
The cases $d_{111}=3,4$ are discharged, because there the origin
lies outside of the lineality part $A_{X,0}^c$.
Moreover, if $d_{111}=0$ holds,
then $d_{211}=1$ must hold because of terminality.
We look at the lattice polytope $C$ in $\QQ^3$
given as the convex hull of the following points:
$$
(l_{21},d_{121},d_{221}), \quad (-2,d_{111},d_{211}), \quad (-2,d_{111}+2,d_{211}),
\quad (-1,0,0), \quad (-1,1,0).
$$
Now, put the leaf $A_X^c \cap \tau_2$ of the 
anticanonical complex into $\QQ^3$ by removing the 
first coordinate (which always equals zero) 
from its points.
For $a = 0,-1,-2$, consider 
$$
H_{a}^+ \ := \{(x,y,z); \; x \ge a\} \ \subseteq \ \QQ^3,
\qquad
H_{a}^0 \ := \{(x,y,z); \; x = a\} \ \subseteq \ \QQ^3.
$$
Then $C \cap H_{0}^+$ equals $A_X^c \cap \tau_2$
and $H_0^0$ cuts out the lineality part $A_{X,0}^c$.
In particular, by terminality of $X$ and Theorem~\ref{thm:main},
the intersection $C \cap H_{0}^+$ has no interior lattice point
and inside $C \cap H_{0}^0$ the origin is the only lattice point.
Other interior lattice points of $C$ may only appear in $C\cap H_{-1}^0$.
For any given pair $(d_{111},d_{112})$ out of the finite set of possible pairs
we find a simplex $B\subseteq C$ containing exactly one interior lattice point
and bound its volume using~\cite[Thm.~2.2]{AvKrNi}.
This technique is the same as the one used in the proof of the 
previous Proposition.
This allows to bound $l_{21}$ according to the table of case (c).

Now we prove~(ii). By Lemma~\ref{lem:tower221bounds8} we have $l_{12}$=1,
therefore we can always achieve $d_{112}=d_{212}=0$ by admissible operations.
The same Lemma gives us $l_{11}=1$ if $l_{21} \ge 5$ or if $h(g_2)\le -4$.
This case is covered by Proposition~\ref{prop:l11-l12-1}.
Let us therefore assume $l_{21}\in\{3,4\}$ and $h(g_2) > -4$, 
together with $l_{11}>1$.
Then Lemma~\ref{lem:tower221bounds8} implies $l_{11}=2$.
Moreover Remark~\ref{rem:tower221bounds4} provides
estimates on $d_{211}$ and $d_{221}$ in terms of $l_{21}$.
The last bounds on $d_{111}$ and $d_{121}$ are obtained as 
in Remark~\ref{rem:tower221bounds0}.

Lastly we turn to~(iii). We have $l_{21} \ge 3$ and $h(g_1) \ge 1$.
Then by Lemma~\ref{lem:tower221bounds4} we must have $l_{11} = l_{12} = 1$.
Hence Proposition~\ref{prop:l11-l12-1} applies.
\end{proof}

Concerning the remaining cases of Lemma~\ref{lem:piconeconfigs},
one shows with arguments similar to those used for
Proposition~\ref{prop:higherplt}
that~(iii),~(v),~(vii) and~(viii) do not provide terminal varieties.
For the cases~(ii),~(iv) and~(vi) we state without proof 
the bounds we obtained.
The arguments are similar as in case~(i)
and are presented in full in~\cite[Section~2.4]{Ni}.

\begin{proposition}
\label{prop:311}
Let $X = X(A,P)$ be a non-toric terminal
$\QQ$-factorial Fano threefold
of Picard number one such that $P$ is irredundant and we have
$r=2$, $m=0$ and $n=5$, where $n_0=3$, $n_1=n_2=1$.
Then $l_{01}=l_{02}=l_{03}=1$ hold and
after suitable admissible operations the matrix $P$ 
is of the form
$$ 
P 
\ = \ 
\left[
\begin{array}{rrrrr}
-1 & -1 & -1 & l_{11} & 0
\\
 -1 & -1 & -1 & 0 & l_{21} 
\\
0 & 1 & 0 & d_{111} & d_{121}
\\
0 & 0 & 1 & d_{211} & d_{221} 
\end{array}
\right],
$$
where $l_{11} \ge l_{21}$ holds. In this setting, we have 
$2 \le l_{21} \le 5$ and we are left with the following 
situations: 
\begin{enumerate}
\item
We have $l_{21} = 2$.
Then we achieve $d_{121}=1$ by suitable admissible operations
and we are in one of the following two cases:
\begin{enumerate}
\item
$d_{221}=0$, $l_{11}\le 69$ hold and we have the estimates
$$
\qquad\quad  -\frac{l_{11}}{2}-1 < d_{211} < 0, 
\qquad  -l_{11} \le d_{111} < -\frac{l_{11}}{2},
$$
\item
$d_{221}=1$, $-35 \le d_{111} <0$ hold and we have the estimates
$$
\qquad\quad  d_{111} \le d_{211} < 0, 
\qquad  \max(2,-d_{111}) \le l_{11} < -2d_{211}.
$$
\end{enumerate}
\item
We have $l_{21} = 3$. Then we achieve $0\le d_{121} \le d_{221} < 3$
by suitable admissible operations and the value $l_{11}$ is bounded 
according to the table
\begin{center}
\begin{tabular}{r|r|r|r|r|r}
$d_{121}$ & $0$ & $0$ & $1$ & $1$ & $2$
\\
$d_{221}$ & $1$ & $2$ & $1$ & $2$ & $2$
\\
\hline
$l_{11}$ & $\le 71$ & $\le 211$ & $\le 103$ & $\le 211$ & $\le 69$
\end{tabular}
\end{center}
and for the remaining entries we obtain the estimates
$$
-\frac{l_{11}}{3}(d_{121}+1)-1 < d_{111} < -\frac{l_{11}}{3}d_{121},
$$
$$
-\frac{l_{11}}{3}(d_{221}+1)-1 < d_{211} < -\frac{l_{11}}{3}d_{221}.
$$
\item
We have $l_{21} = 4$ or $l_{21} = 5$. Then we have following estimates
$$
\qquad\qquad l_{11} < \frac{3l_{21}}{l_{21}-3},
\qquad\quad 0\le d_{121},d_{221}< l_{21},
$$
$$
\qquad\quad
-\frac{l_{11}d_{121}}{l_{21}} - l_{11} < d_{111} < -\frac{l_{11}d_{121}}{l_{21}},
\qquad
-\frac{l_{11}d_{221}}{l_{21}} - l_{11} < d_{211} < -\frac{l_{11}d_{221}}{l_{21}}.
$$
\end{enumerate}
\end{proposition}

\begin{proposition}
\label{prop:2211}
Let $X = X(A,P)$ be a non-toric terminal
$\QQ$-factorial Fano threefold
of Picard number one such that $P$ is irredundant and we have
$r=3$, $m=0$ and $n=6$, where $n_0=n_1=2$, $n_2=n_3=1$.
Then $l_{01}=l_{02}=l_{11}=l_{12}=1$ hold and
after suitable admissible operations the matrix $P$ is of the form
$$ 
P 
\ = \ 
\left[
\begin{array}{rrrrrr}
-1 & -1 & 1 & 1 & 0 & 0
\\
 -1 & -1 & 0 & 0 & l_{21} & 0
\\
 -1 & -1 & 0 & 0 & 0 & l_{31} 
\\
0 & 1 & d_{111} & 0 & d_{121} & d_{131}
\\
0 & 0 & d_{211} & 0 & d_{221} & d_{231}
\end{array}
\right],
$$
such that $l_{21} \ge l_{31}$ holds. 
In this setting $l_{31}=2,3$ holds
and we are left with the following situations:
\begin{enumerate}
\item
We have $l_{31} = 2$.
Then we have $d_{211}=1$ and we can achieve $d_{111}=0$ 
by a suitable admissible operation.
The other entries of $P$ are then bounded according 
to the table
\begin{center}
\begin{tabular}{r|r|r|r}
$d_{131}$ & $0$ & $1$ & $1$
\\
$d_{231}$ & $1$ & $0$ & $1$
\\
\hline
$l_{21}$ & $\le 33$ & $\le 141$ & $\le 69$
\end{tabular}
\end{center}
and the estimates
$$
-\frac{l_{21}}{2}(d_{231}+1)-1 < d_{221} < -\frac{l_{21}}{2}d_{231},
$$
$$
-\frac{d_{221}}{l_{21}}-\frac{d_{231}}{2} < d_{211} <
-\frac{d_{221}}{l_{21}}-\frac{d_{231}}{2}+ \frac{2+l_{21}}{2l_{21}}.
$$
\item
We have $l_{31} = 3$. Then $3 \le l_{21} \le 5$ holds,
we obtain $0\le d_{131},d_{231} < 3$ and we have the estimates
$$
-\frac{l_{21}}{3}(d_{231}+1)-1 < d_{221} < -\frac{l_{21}}{3}d_{231},
$$
$$
-\frac{d_{221}}{l_{21}}-\frac{d_{231}}{3} < d_{211} <
-\frac{d_{221}}{l_{21}}-\frac{d_{231}}{3}+ \frac{3+l_{21}}{3l_{21}},
$$
$$
0 \le d_{111} < d_{211}l_{31}
$$
$$
\frac{d_{111}(l_{21}d_{231}+3d_{221})-l_{21}d_{211}d_{131}}{3d_{211}}-l_{21} < d_{121} <
\frac{d_{111}(l_{21}d_{231}+3d_{221})-l_{21}d_{211}d_{131}}{3d_{211}}.
$$
\end{enumerate}
\end{proposition}

\begin{proposition}
\label{prop:211}
Let $X = X(A,P)$ be a non-toric terminal
$\QQ$-factorial Fano threefold
of Picard number one such that $P$ is irredundant and we have
$r=2$, $m=1$ and $n=4$, where $n_0=2$, $n_1=n_2=1$.
Then $l_{01}=l_{02}=1$ hold and after suitable admissible 
operations the matrix $P$ is of the form
$$ 
P 
\ = \ 
\left[
\begin{array}{rrrrr}
-1 & -1 & l_{11} & 0 & 0
\\
 -1 & -1 & 0 & l_{21}  & 0
\\
0 & 1 & d_{111} & d_{121} & d_{11}'
\\
0 & 0 & d_{211} & d_{221} & d_{21}'
\end{array}
\right],
$$
where $2 \le l_{21}\le l_{11}$, $0\le d_{121},d_{221} <l_{21}$ and $0\le d_{11}' < d_{21}'$ hold.
In this situation, one has the estimates
$$
-\frac{l_{11}}{l_{21}}d_{121} -l_{21}
+\frac{d_{11}'}{d_{21}'}\Bigl(d_{211}+\frac{l_{11}}{l_{21}}d_{221}\Bigr)
\ < \ d_{111} \ < \ -\frac{l_{11}}{l_{21}}d_{121},
$$
$$
-\frac{l_{11}}{l_{21}}(d_{221}+2)-2  \ < \ d_{211} \ < \ -\frac{l_{11}}{l_{21}}d_{221}.
$$
Moreover, we are in one of the following situations:
\begin{enumerate}
\item
we have $d_{21}'=1$. Then $d_{11}'=0$ and $l_{21} \le 7 $ hold
and $l_{11}$ is bounded according to the table
\begin{center}
\begin{tabular}{r|r|r|r|r|r|r}
$l_{21}$ & $2$ & $3$ & $4$ & $5$ & $6$ & $7$
\\
\hline
$l_{11}$ & $\le 51$ & $\le 105$ & $\le 11$ & $\le 19$ & $\le 11$ & $\le 9$
\end{tabular}  .
\end{center}
\item
we have $d_{21}'>1$. Then $d_{11}'>0$ and $l_{21}\le5$ hold and 
we are in one of the following subcases:
\begin{enumerate}
\item
$l_{21}=2$, $l_{11}\le69$, $d_{21}'=2,\ldots,10$ and $d_{11}'\in\{1,d_{21}'-1\}$.
\item
$l_{21}=3$, $(d_{11}',d_{21}') \in \{(1,2),(1,3),(2,3),(3,4)\}$
and the exponent $l_{11}$ is bounded according to the table:
\begin{center}
\begin{tabular}{r|r|r|r|r}
$(d_{11}',d_{21}')$ & $(1,2)$ & $(1,3)$ & $(2,3)$ & $(3,4)$
\\
\hline
$l_{11}$ & $\le 14$ & $\le 4$ & $\le 4$ & $= 3$
\end{tabular}  .
\end{center}
\item
$l_{21}=4,5$, $l_{11}\le 11$, $(d_{11}',d_{21}') = (1,2)$.
\end{enumerate}
\end{enumerate}
\end{proposition}

\begin{remark}
Propositions~\ref{prop:l11-l12-1} to~\ref{prop:211} 
give us effective bounds on the entries of the defining matrices 
$P$ for the terminal $\QQ$-factorial
Fano threefolds $X = X(A,P)$ with effective 
two-torus action and Picard number $\varrho(X) = 1$.
In order to prove Theorem~\ref{thm:classif}
one still has to figure out the terminal ones among
all candidates $X = X(A,P)$, where $P$ fulfills these bounds.
This means to check Condition~\ref{thm:main}~(v);
we do it by computer, using~\cite{MDS} where the anticanonical
complex $A_X^c$ is implemented.
Using the explicit knowledge of canonical Fano 
$3$-topes provided by Kasprzyk's classification~\cite{Ka2}, 
one can reduce the number of testing cases and 
obtains more specific bounds in the remaining cases.
\end{remark}

\begin{remark}
If one adds the assumption ``$\Cl(X)$ finitely generated''
in Theorem~\ref{thm:classif}, then, without further changes,
all the results and proofs of the paper are valid over any 
algebraically closed field of characteristic zero.
\end{remark}


\end{document}